\documentclass[11pt,reqno]{amsart}

\usepackage[utf8]{inputenc}
\usepackage[T1]{fontenc}

\usepackage{amsmath, amssymb, amsthm}
\usepackage{geometry}
\usepackage{enumitem}
\usepackage{hyperref}
\hypersetup{
  colorlinks,
  citecolor=blue,
  linkcolor=blue,
  urlcolor=blue}
\usepackage{marginnote}
\usepackage{xcolor}
\usepackage{upgreek}
\usepackage{tikz}
\usetikzlibrary{calc}
\usepgflibrary {shadings} 
\usetikzlibrary{intersections}
\usetikzlibrary{patterns}
\usepackage{mathtools}
\usepackage{verbatim}
\usepackage{lipsum}  
\usepackage{todonotes}
\usepackage{mathrsfs}

\definecolor{bleudefrance}{rgb}{0.19, 0.55, 0.91}
\newcommand{\colorset}{bleudefrance}

\numberwithin{equation}{section}

\theoremstyle{plain}
\newtheorem{theorem}{Theorem}[section]
\newtheorem{lemma}[theorem]{Lemma}

\newtheorem{corollary}[theorem]{Corollary}

\theoremstyle{definition}
\newtheorem{definition}[theorem]{Definition}

\newtheorem{remark}[theorem]{Remark}
\newtheorem{example}[theorem]{Example}

\newcommand{\C}{\mathbb{C}}
\newcommand{\R}{\mathbb{R}}

\newcommand{\s}{\mathbb{S}}
\newcommand{\A}{\mathcal{A}}
\renewcommand{\H}{\mathcal{H}}
\renewcommand{\div}{\operatorname{div}}
\newcommand{\Per}{\operatorname{Per}}
\newcommand{\pv}{\operatorname{p.v.}}

\usepackage[style=alphabetic, giveninits=true]{biblatex}
\begin{document}

\title{Nonlocal Free Boundary minimal surfaces}

\author[M. Badran]{Marco Badran}
	\address{ETH Z\"urich, Department of Mathematics, Rämistrasse 101, 8092 Zürich, Switzerland.}
 	\email{marco.badran@math.ethz.ch}
 \author[S. Dipierro]{Serena Dipierro}
 	\address{Department of Mathematics and Statistics, The University of Western Australia, 35 Stirling Highway, Perth, WA 6009, Australia}
 	\email{serena.dipierro@uwa.edu.au}
 \author[E. Valdinoci]{Enrico Valdinoci}
	\address{Department of Mathematics and Statistics, The University of Western Australia, 35 Stirling Highway, Perth, WA 6009, Australia}
	\email{enrico.valdinoci@uwa.edu.au}

\maketitle

\begin{abstract}
	We introduce the nonlocal analogue of the classical free boundary minimal hypersurfaces in an open domain~$\Omega$ of~$\R^n$ as the (boundaries of) critical points of the fractional perimeter~$\operatorname{Per}_s(\cdot,\,\Omega )$ with respect to inner variations leaving~$\Omega$ invariant. We deduce the Euler--Lagrange equations and prove a few surprising features, such as the existence of critical points without boundary and a strong volume constraint in~$\Omega$ for unbounded hypersurfaces. Moreover, we investigate stickiness properties and regularity across the boundary.
\end{abstract}

\tableofcontents

\section{Introduction}

\subsection{Nonlocal free boundary minimal surfaces}

The notion of ``free boundary minimal surfaces'' arises quite naturally in many problems of physical interest. While an area-minimising surface is typically a surface that minimises area for given boundary data (producing surfaces with zero mean curvature that attach at a prescribed curve along the boundary), in the free boundary case the boundary of the surface is not fixed, and it can ``float'' to adjust itself in order to minimise the surface area subject to certain constraints. For example, a droplet in a container (or a cell membrane in a biological organism) tends to minimise surface tension for a given volume: in this case, the boundary of the droplet is not prescribed and forms a contact angle with the container which is of practical importance.

{F}rom the mathematical point of view, the analysis of free boundary minimal surfaces dates back to Richard Courant~\cite{MR2478}, who considered the problem of minimising the area when the boundaries are ``free to move on prescribed manifolds'' and showed that the minimal surface obtained in this way meets the prescribed manifold orthogonally.

More precisely, free boundary minimal hypersurfaces in a domain~$\Omega$ are defined as
critical points of the area functional with respect to variations supported up to the boundary of~$\Omega$, generated by a flow that leaves~$\partial\Omega$ invariant.
	There is a vast literature related to free boundary minimal surfaces, with a special attention to the case of~$\Omega$ being the unit ball of~$\R^3$, see e.g. the survey~\cite{Li2020}.
	\medskip

In this article, we extend the notion of free boundary minimal surface to the nonlocal setting and we study the
basic properties of these objects, also discovering some quite surprising facts.
\medskip

Nonlocal minimal surfaces were first introduced in~\cite{Caffarelli-Roquejoffre-Savin2010} as minimisers of an integral energy, related to the interfaces of phase transition models accounting for long-range interactions, see e.g.~\cite{MR4581189}.

Given an open set~$\Omega\subset\R^n$ and a (measurable) set~$E\subset \R^n$, we define for~$s\in(0,1)$ the~$s$-fractional perimeter of~$E$ in~$\Omega$ as
\begin{equation*}
	\Per_s(E;\Omega)\coloneqq c_{n,s}\left(\int_{E^c\cap \Omega}\int_{E\cap\Omega}+\int_{E\cap \Omega}\int_{E^c\cap\Omega^c}+\int_{E\cap \Omega^c}\int_{E^c\cap\Omega}\right)\frac{dxdy}{|x-y|^{n+s}},
\end{equation*}
where~$c_{n,s}$ is a renormalisation constant, given explicitly in formula~\eqref{eq: cns} below.

Here above and in the rest of the paper, the notation~$E^c$ is used to denote the complementary set of~$E$, namely~$E^c\coloneqq \R^n\setminus E$.

\medskip

An~$s$-minimal hypersurface in~$\Omega$ is (the boundary of) a critical point of~$\Per_s(\cdot\,;\Omega)$ with respect to variations compactly supported in~$\Omega$. It is well known that critical points satisfy weakly the Euler--Lagrange equation 
\begin{equation*}
	\H^s_{E}(x)\coloneqq c_{n,s}\pv \int_{\R^n}\frac{\chi_{E^c}(y)-\chi_E(y)}{|x-y|^{n+s}}dy=c_{n,s}\lim_{\delta\to 0}\int_{\R^n\setminus B_\delta(x)}\frac{\chi_{E^c}(y)-\chi_E(y)}{|x-y|^{n+s}}dy=0
\end{equation*} 
for~$x\in\partial E\cap \Omega$, where~``$\pv$'' is a standard abbreviation for wording
``in the Cauchy principal value sense'', see e.g.~\cite{Figalli-Fusco-Maggi-Millot-Morini2015}.
\medskip

Ever since their introduction, these objects have been subject of intensive investigation, concerning both their existence and regularity properties~\cite{Caffarelli-Valdinoci2013,Savin-Valdinoci2013,Davila-delPino-Wei2018,Cinti-Serra-Valdinoci2019,Cabre-Cinti-Serra2020}, as well as convergence to the classical perimeter~\cite{MR3586796,MR1942130,Caffarelli-Valdinoci2011,Ambrosio-DePhilippis-Martinazzi2011,Florit2024}. Recently, stable and finite-index critical points of the fractional perimeter were studied in the context of closed Riemannian manifolds~\cite{Florit2024,Caselli-Florit-Serra2024}, obtaining, among other things, a nonlocal analogue of Yau's conjecture in dimension~$3$~\cite{Caselli-Florit-Serra2024Yau}. 

\medskip

Perhaps the most natural question one could ask after defining~$s$-minimal surfaces is a nonlocal analogue of Plateau's problem: given a certain set~$E'\subset \Omega^c$, does it exist a set~$E$ such that~$E$ minimises~$\Per_s(\cdot\,,\Omega)$ among all sets satisfying~$E\equiv E'$ in~$\Omega^c$?

This question was already answered positively in the original paper by Caffarelli, Roquejoffre and Savin~\cite{Caffarelli-Roquejoffre-Savin2010} using the direct method of calculus of variations. 
\medskip

In some sense, Plateau's problem can be interpreted as the Dirichlet problem for the nonlocal minimal surface equation. The goal of the present paper is to introduce the natural Neumann counterpart, namely \emph{free boundary~$s$-minimal hypersurfaces}. 
\medskip

We define free boundary~$s$-minimal hypersurfaces to be boundaries of critical points of~$\Per_s(\cdot\,,\Omega)$ with respect to inner variations compactly supported in~$\R^n$ (not necessarily in~$\Omega$) and leaving~$\partial\Omega$ invariant, namely the variations generated by a vector field which is tangent to~$\partial\Omega$ at any of its points.
\medskip

In particular, we will show that the Euler--Lagrange equations require that critical points satisfy weakly the~$s$-minimal hypersurface equation in~$\Omega$
\begin{equation}\label{eq: s-min surf}
	c_{n,s}\int_{\R^n}\frac{\chi_{E^c}(y)-\chi_{E}(y)}{|x-y|^{n+s}}dy=0,\quad\quad {\mbox{for all }} x\in \partial E\cap\Omega
\end{equation}
and the nonlocal free boundary condition outside of~$\Omega$
\begin{equation}\label{eq: free boundary}
	c_{n,s}\int_{\Omega}\frac{\chi_{E^c}(y)-\chi_{E}(y)}{|x-y|^{n+s}}dy=0,
	\quad\quad  {\mbox{for all }} x\in \partial E\cap\overline\Omega^c.
\end{equation}

While the quantity in~\eqref{eq: s-min surf} is typically referred to as the nonlocal mean curvature (or~$s$-mean curvature) of the set~$E$, the one in~\eqref{eq: free boundary} is a new object accounting
for nonlocal interactions of a given point outside~$\Omega$ with points in the reference set~$\Omega$.

Correspondingly, while~\eqref{eq: s-min surf} is the Euler--Lagrange equation
of the~$s$-perimeter functional with Dirichlet datum,
the Neumann counterpart will present equation~\eqref{eq: free boundary} as an additional
prescription, with the interesting feature that
nonlocal free boundary minimal surfaces
satisfy the nonlocal mean curvature equation~\eqref{eq: s-min surf}
along all boundary points inside the reference set~$\Omega$
and the ``Neumann condition''~\eqref{eq: free boundary}
outside the reference set~$\Omega$, in a sense that we are now making precise.

\subsection{The Euler--Lagrange equation of nonlocal free boundary minimal surfaces}
Let $\Omega\subset \R^n$ be a set with $C^1$ boundary and consider a vector field~${\mathcal{X}}\in C^\infty_c(\R^n,\R^n)$ with the property that 
\begin{equation}\label{eq: tangency}
	{\mathcal{X}}(x)\in T_x(\partial\Omega),\quad {\mbox{for all }} x\in\partial\Omega.
\end{equation}
Let~$\Phi_t$ be the flow associated with~${\mathcal{X}}$, namely the solution to~$\partial_t\Phi_t={\mathcal{X}}\circ \Phi_t$ with initial condition~$\Phi_0=\operatorname{id}$. Given a set~$E$, denote by~$E_t\coloneqq \Phi_t(E)$. 

\begin{definition}[Nonlocal free boundary minimal surfaces] Let~$E\subset\R^n$. We say that~$\partial E$
is a nonlocal free boundary minimal hypersurface in $\Omega$ (or, with a slight abuse of notation, that~$
E$ is a nonlocal free boundary minimal hypersurface) if, for all~${\mathcal{X}}\in C^\infty_c(\R^n,\R^n)$
satisfying~\eqref{eq: tangency},
\begin{equation}\label{eq: criticality}
	\frac{d}{dt}\Big\vert_{t=0}\Per_s(E_t;\Omega)=0.
\end{equation}\end{definition}

We will compute~\eqref{eq: criticality} in a slightly more general setting, where the fractional kernel~$c_{n,s}|z|^{-n-s}$ is replaced by a possibly anisotropic kernel. 

For this, we say that a kernel~$K\in C^{1}(\R^n\setminus\{0\},[0,+\infty))$ is an
\emph{admissible~$s$-kernel} if it is symmetric around the origin (namely,~$K(z)=K(-z)$) and satisfies the bound 
\begin{equation}\label{nvcmxwkdefaddmiss}
	K(z)\leq \frac{C_K}{|z|^{n+s}}
\end{equation}
for some constant~$C_K>0$. 

For~$s\in(0,1)$ and for any admissible~$s$-kernel~$K$, we set 
\begin{eqnarray*}
	\Per_K(E;\Omega)&\coloneqq& \left(\int_{E^c\cap \Omega}\int_{E\cap\Omega}+\int_{E\cap \Omega}\int_{E^c\cap\Omega^c}+\int_{E\cap \Omega^c}\int_{E^c\cap\Omega}\right)K(x-y)dxdy\\
	&\eqqcolon& \mathcal{I}_K(E^c\cap \Omega,E\cap\Omega)+\mathcal{I}_K(E\cap \Omega,E^c\cap\Omega^c)+\mathcal{I}_K(E\cap \Omega^c,E^c\cap\Omega).
\end{eqnarray*}
We also denote, for all~$x\in\partial E$,
\begin{equation*}
	\H^K_{E}(x)\coloneqq\pv \int_{\R^n}\big(\chi_{E^c}(y)-\chi_{E}(y)\big)K(x-y)dy
\end{equation*}
and, for all~$x\in \partial E\cap \overline\Omega^c$,
\begin{equation*}
	\A^K_{E,\Omega}(x)\coloneqq\int_{\Omega}\big(\chi_{E^c}(y)-\chi_{E}(y)\big)K(x-y)dy.
\end{equation*}
This setting can be considered as a generalization of
the nonlocal mean curvature in~\eqref{eq: s-min surf}
and the free boundary condition in~\eqref{eq: free boundary}
to possibly anisotropic kernels.

Given~$\alpha\in[0,1]$, we say that a set~$U\subset\R^n$ is of class~$C^{1,\alpha}$
if there exist~$\rho$, $M>0$ such that for every~$p\in\partial U$
the set~$U\cap B_\rho(p)$ can be written as the subgraph, in some direction,
of a function of class~$C^{1,\alpha}$ and with~$C^{1,\alpha}$-norm bounded by~$M$.


Furthermore, 
given~$E$, $\Omega\subset\R^n$ of class~$C^{1}$, we
say that they intersect uniformly transversally if there exists a constant~$\mu\in(0,1)$ such that 
\begin{equation}\label{eq: unif transversality}
	\sup_{q\in\partial \Omega\cap\partial E}|\nu_{\partial\Omega}(q)\cdot \nu_{\partial E}(q)|\leq 1-\mu.
\end{equation}

With this notation, the Euler--Lagrange equation reads as follows.

\begin{theorem}\label{PROP1}
Let~$s\in(0,1)$ and~$K$ be an admissible~$s$-kernel.
Let~$\Omega$ be a set of class~$C^{1}$ and~$E$ be a set of class~$C^{1,\alpha}$ for some~$\alpha\in(s,1)$. Assume also that~$\Omega$ and~$E$ intersect uniformly transversally.

Let~${\mathcal{X}}\in C^{\infty}_c(\R^n,\R^n)$ be a vector field satisfying~\eqref{eq: tangency} and let~$\Phi_t$ be the flow generated by~${\mathcal{X}}$. 
	
	Then,
	\begin{equation}\label{eq: first variation}
		\begin{split}&
			\frac{d}{dt}\Big\vert_{t=0}\Per_K(\Phi_t(E);\Omega)\\&\quad=\int_{\partial E\cap\Omega}H^K_{E}(x){\mathcal{X}}(x)\cdot\nu_{\partial E}(x)d\mathscr{H}^{n-1}_x+\int_{\partial E\cap\Omega^c}\A^K_{E,\Omega}(x){\mathcal{X}}(x)\cdot\nu_{\partial E}(x)d\mathscr{H}^{n-1}_x.
		\end{split}
	\end{equation}
	\end{theorem}
	
The proof of Theorem~\ref{PROP1} relies on the
strategy developed in~\cite[Theorem~6.1]{Figalli-Fusco-Maggi-Millot-Morini2015}, namely one first regularises the kernel and then shows that the result obtained passes to the limit.

To implement the passage to the limit, a pivotal step consists in  
understanding the asymptotic behaviour of~$\A_{E,\Omega}^K$ near~$\partial\Omega$. For this, given~$\delta>0$, we consider a~$\delta$-tubular neighbourhood~$T_{\delta}$ of~$\partial\Omega$ defined as
$$ T_{\delta} \coloneqq \bigcup_{x\in\partial\Omega} B_\delta(x).
$$
The result needed for our purposes reads as follows:

\begin{lemma}\label{lem: estimate K near Omega}
	Let~$s\in(0,1)$ and~$K$ be an admissible~$s$-kernel.
	Let~$\Omega$ be an open set in~$\R^n$ of class~$C^{1}$.
	
	Let $\bar\delta>0$ be small enough so that 
	the nearest point projection from~$T_{\bar\delta}$ to~$\partial\Omega$ is well defined.
	
Then, there exists~$\delta_0\in(0,\bar\delta)$ such that, for every~$\delta\in(0,\delta_0]$ and
for every~$x\in T_{\delta}\cap\overline\Omega^c$,
	\begin{equation*}
		\int_\Omega K(x-y)dy\leq C\operatorname{dist}(x,\Omega)^{-s}
	\end{equation*}
	for some~$C>0$ depending only on~$C_K$, $n$ and~$s$.
\end{lemma}

We will prove Theorem~\ref{PROP1} and
Lemma~\ref{lem: estimate K near Omega} in Section~\ref{sec: EL} below.

\begin{example}
	The simplest nontrivial example of a free boundary~$s$-minimal surface is the hyperplane~$\partial\{y_n<0\}$ in the unit ball. The interior equation~\eqref{eq: s-min surf} is satisfied in every compact set, while the free boundary condition~\eqref{eq: free boundary} holds by symmetry. 
\end{example}

\subsection{The free boundary condition in the limit as~$s\nearrow 1$}
As a natural next step, we show that the nonlocal free boundary condition ``converges'', in some suitable sense, to the local one, as~$s\nearrow 1$. 

For this, let us recall that the constant~$c_{n,s}$ is explicitly given by 
\begin{equation}\label{eq: cns}
	c_{n,s}\coloneqq \frac{2^{2+2s}\Gamma\left(\frac{n+s}2\right)}{\pi^{n/2}\Gamma(2-s)}s(1-s)
\end{equation}
and we have the limits
\begin{equation}\label{eq: cns to 1}
	\lim_{s\searrow 0}\frac{c_{n,s}}{s}=\frac{8}{\omega_n}\qquad{\mbox{and}} \qquad
	\lim_{s\nearrow 1}\frac{c_{n,s}}{1-s}=\frac{16n}{\omega_n},
\end{equation}
see, for instance, \cite[\S1.6.2]{Abatangelo-Dipierro-Valdinoci2025}.
Here above~$\omega_n$ denotes the surface measure of the~$(n-1)$-dimensional sphere~$\partial B_1\subset\R^n$.
\medskip

We show that the free boundary condition~\eqref{eq: free boundary}, defined on~$\partial E\cap\Omega^c$, concentrates on~$\partial E\cap\partial\Omega$ as~$s\nearrow 1$, according to the following statement:

\begin{theorem}\label{prop: s to 1}
Let~$E$ and~$\Omega$ be open sets of class~$C^{1,1}$ intersecting uniformly transversally.
Let~$\mathcal{X}\in C^{\infty}_c(\R^n,\R^n)$ be a vector field satisfying~\eqref{eq: tangency}.

Then,
	\begin{equation*}
		\lim_{s\nearrow1}\int_{\partial E\cap \Omega^c}\A_{E,\Omega}^s(x)\mathcal{X}(x)\cdot \nu_{\partial E}(x)d\mathscr{H}^{n-1}_x= \int_{\partial E\cap \partial \Omega}g(\psi_{x'})\mathcal{X}(x')\cdot \nu_{\partial E}(x')d\mathscr{H}^{n-2}_{x'},
	\end{equation*}
	where
	\begin{equation*}
	g(\psi)\coloneqq \int_{{\{Z_1<0\}}\atop{
\{Z_n\in((Z_1-1)/\tan\psi,(1-Z_1)/\tan\psi
)\}}}\frac{dZ}{|Z-e_1|^{n+1}} 
\end{equation*}
and $\psi_{x'}$ is the intersection angle between the affine hyperplanes~$T_{x'}(\partial E)$ and~$T_{x'}(\partial \Omega)$. 
\end{theorem}

An immediate consequence of Theorem \ref{prop: s to 1} is the following.

\begin{corollary}
	Under the same assumptions of Theorem \ref{prop: s to 1}, 
	\begin{equation*}
		\lim_{s\nearrow1}\int_{\partial E\cap \Omega^c}\A_{E,\Omega}^s(x)\mathcal{X}(x)\cdot \nu_{\partial E}(x)d\mathscr{H}^{n-1}_x=0
	\end{equation*}
	if and only if $\partial E$ meets $\partial\Omega$ orthogonally in the support of~$\mathcal{X}$.
\end{corollary}

We refer to Section~\ref{sec: lim s1} below for the proof of Theorem~\ref{prop: s to 1}. 

\subsection{A free boundary $s$-minimal surface \emph{without} free boundary}
In the classical case, if the reference domain~$\Omega$ is bounded, the boundary of a smooth free boundary minimal surface~$E$ must always meet the boundary of~$\Omega$ (unless~$E$ is either void or contains the whole domain). Indeed, otherwise, assuming that~$0\in\Omega$, one could pick a point~$p\in\partial E$ that maximises the distance from the origin and note that either~$E\subseteq B_r$
or~$E^c\subseteq B_r$, with~$r\coloneqq |p|$. The mean curvature of~$\partial E$ at~$p$ would then be bounded below by~$\frac1r$ (when~$E\subseteq B_r$) or above by~$-\frac1r$ (when~$E^c\subseteq B_r$), thus contradicting the zero mean curvature condition.

Alternatively, one can observe that there are no closed minimal hypersurfaces~$\Sigma$ in~$\R^n$, since
the coordinate functions~$x_1,\dots,x_n\colon\Sigma\to\R$ are harmonic on~$\Sigma$
and therefore constant if~$\Sigma$ has no boundary.
This entails that a free boundary minimal hypersurface in a compact manifold~~$\Omega$ always meets the boundary.

In stark contrast  with the local case, we show
that, in the asymptotic regime~$s\sim 0$, there exist free boundary~$s$-minimal hypersurfaces without any boundary contact set.
Moreover, such surfaces are radially symmetric, which is again an exclusively nonlocal feature. The precise statement is the following:

\begin{theorem}\label{thm: rad symm sFBMS}
	There exists~$s_\circ\in(0,1)$ such that for every~$s\in(0,s_\circ)$ there exist radii~$0<r_1<1<r_2$ such that~$\partial(B_{r_2}\setminus B_{r_1})$ is a free boundary~$s$-minimal hypersurface in~$B_1$.
\end{theorem}

Theorem~\ref{thm: rad symm sFBMS} will be established in Section~\ref{sec: FBWFB}.

\subsection{The volume condition for unbounded sets}
We also prove a remarkably simple and powerful property of free boundary~$s$-minimal hypersurfaces: whenever~$\partial E$ is unbounded and~$\Omega$ is bounded, the volume of~$E$ in~$\Omega$ must coincide with the volume of~$E^c$ in~$\Omega$.
 
\begin{theorem}\label{thm: volume constraint}
Let~$s\in(0,1)$, $\Omega$ be an open, bounded set and~$E$ be a free boundary~$s$-minimal surface with~$\partial E$ unbounded.

	Then, 
	\begin{equation}\label{eq: volume condition}
		\mathscr{H}^n(E\cap\Omega)=\mathscr{H}^n(E^c\cap\Omega).
	\end{equation}
	\end{theorem}
	
The proof of Theorem~\ref{thm: volume constraint} is contained in Section~\ref{sec: volume}.
	
\begin{remark}
	The assumption of~$\partial E$ being unbounded
	in Theorem~\ref{thm: volume constraint} is necessary, as can be seen from the example of Theorem~\ref{thm: rad symm sFBMS}. 
\end{remark}

The volume condition~\eqref{eq: volume condition} can be used to show that many natural candidates for the nonlocal analogues of classical free boundary minimal surfaces are not actually free boundary~$s$-minimal surfaces for most (if not any) $s\in(0,1)$. 

\medskip

Consider the Lawson cones in~$\R^{n+m}$
\begin{equation*}
	C^{n,m}(\alpha)\coloneqq\{(x,y)\in\R^n\times\R^m:|x|<\alpha|y|\}.
\end{equation*}
It was proved in~\cite{Davila-delPino-Wei2018} that for every~$n$, $m\geq 1$ and every~$s\in(0,1)$ there exists a unique~$\alpha=\alpha(s,n,m)>0$ such that~$\partial C^{n,m}(\alpha)$ is a critical point for the~$s$-perimeter. For such~$\alpha$, we denote the solid cone by~$C^{n,m}_s$. We remark that, by symmetry, for any~$n\geq1$ and every~$s\in(0,1)$,
\begin{equation}\label{eq: opening symmetric cones}
	\alpha(s,n,n)=1.
\end{equation}

In~\cite{Davila-delPino-Wei2018} it was also proved that 
\begin{equation}\label{eq: opening C21}
	\alpha(s,2,1)=\sqrt{1-s}+O(1-s)\quad \text{as }s\nearrow 1,
\end{equation}
thus, the opening of the~$s$-critical Lawson cone in~$\R^3$ vanishes as~$s\nearrow 1$. Such cone arises as the blow-down of the fractional catenoid~$F_s$, constructed in~\cite{Davila-delPino-Wei2018}. We recall that, as~$s\nearrow1$, $F_s$ converges to the set~$F_*$ whose boundary is a classical catenoid, uniformly in every compact set.

\medskip

It is reasonable to wonder whether~$s$-critical Lawson cones are free boundary~$s$-minimal hypersurfaces in any ball, since any of these cones trivially satisfies the classical free boundary condition. 

Moreover, a simple limiting argument shows the existence of a ball in which the catenoid is a free boundary minimal surface, thus it is natural to wonder whether the fractional analogues~$F_s$ are free boundary in some ball.

The volume condition~\eqref{eq: volume condition} entails that this is not the case, according to the following result. 

\begin{corollary}\label{CPscf:2er3gr}
There exists~$\overline{s}\in(0,1)$ such that for all~$s\in(\overline{s},1)$ the Lawson cone~$C^{2,1}_s$ and the fractional catenoids~$F_s$ are not free boundary~$s$-minimal surfaces in any ball~$B_R(0)\subset\R^3$.
\end{corollary}

\begin{remark}
By the monotonicity and the continuity of the map 
		\begin{equation*}
			\alpha\to \mathscr{H}^{n+m}(C^{n,m}(\alpha)\cap B_1),
		\end{equation*}
		along with the facts that 
		\begin{equation*}
			\mathscr{H}^{n+m}(C^{n,m}(0)\cap B_1)=0\qquad{\mbox{and}}\qquad \mathscr{H}^{n+m}(C^{n,m}(\infty)\cap B_1)=\mathscr{H}^{n+m}(B_1),
		\end{equation*}
		we infer the existence of a unique~$\alpha$ such that
		\begin{equation}\label{utirw43829gfhew8ir}
		\mathscr{H}^{n+m}(C^{n,m}(\alpha)\cap B_1)=\frac{\mathscr{H}^{n+m}(B_1)}2.\end{equation}
Notice that~\eqref{utirw43829gfhew8ir} corresponds to the volume condition in~\eqref{eq: volume condition} in this setting.

In the case~$n=m$, by symmetry, we also know that~$C^{n,n}(1)$ is a free boundary~$s$-minimal surface in any ball.
Therefore, in this setting, the unique~$\alpha$ satisfying~\eqref{utirw43829gfhew8ir} is~$\alpha=1$,
and all the other cones~$C^{n,n}(\alpha)$, with~$\alpha\neq1$, are not free boundary~$s$-minimal surfaces in~$B_1$.

It would be interesting to classify all the fractional Lawson cones~$C^{m,n}_s$ which are
free boundary~$s$-minimal surfaces also when~$n\neq m$. The fact that the volume condition~\eqref{eq: volume condition}
does not depend on~$s$ suggests that it is ``difficult'' for~$C^{m,n}_s$ to satisfy this condition for a ``generic''~$s$
and one may even conjecture that the only fractional Lawson cones that are free boundary in a ball are the symmetric cones~$C^{n,n}(1)$.
\end{remark}

\subsection{Stickiness and regularity across the boundary}

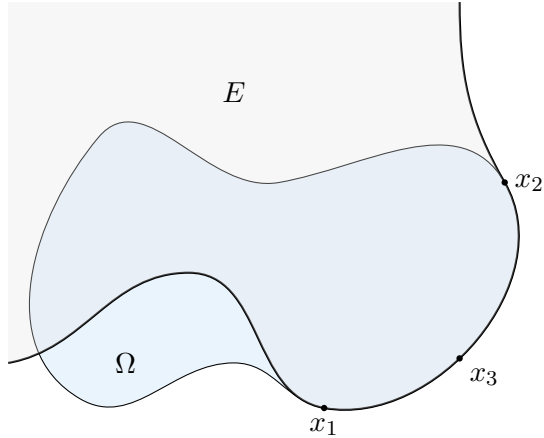
\begin{figure}
	\begin{center}
		\begin{tikzpicture}[scale=1.2]
		
		\coordinate (A) at (0,0);
		\coordinate (B) at (1,-.5);
		\coordinate (C) at (3,2);
		\coordinate (D) at (0.5,2);
		\coordinate (E) at (-1.5,2.5);
		\coordinate (F) at (-1.7,-.4);
		\coordinate (G) at (2.5,0.05);
		
		\coordinate (A1) at (2.5,4);
		\coordinate (A2) at (-.5,1);
		\coordinate (A3) at (-2.5,0);
		\coordinate (A4) at (-2.5,4);


		\filldraw[fill=\colorset!10, draw=black] (A) to[out=0, in=170] (B) to[out=-10,in=-60] (C) to[out=120,in=10] (D) to[out=190,in=50] (E) to[out=230,in=150] (F) to[out=-30,in=180] (A);
		
		\draw[black, thick] (A3) to[out=10,in=180] (A2) to[out=0, in=170] (B) to[out=-10,in=-60] (C) to[out=120,in=-90] (A1);
		
		\fill[fill=gray!20, fill opacity=0.3] (A3) to[out=10,in=180] (A2) to[out=0, in=170] (B) to[out=-10,in=-60] (C) to[out=120,in=-90] (A1) to (A4) to (A3);
		
		\fill (B) circle (1pt);
		\fill (C) circle (1pt);
		\fill (G) circle (1pt);
		
		
		\fill (-1.2,0) node {$\Omega$};
		\fill (0,3) node {$E$};
		\fill (B) node[below] {$x_1$};
		\fill (C) node[right] {$x_2$};
		\fill (G) node[below right] {$x_3$};
		
		\end{tikzpicture}
	\end{center}
	\caption{Different types of stickiness: from inside on~$x_1$, from outside on~$x_2$ and bilateral on~$x_3$.}
	\label{fig: stickiness}
\end{figure}

The ``stickiness'' phenomenon was introduced in~\cite{Dipierro-Savin-Valdinoci2017} and describes the (generic, see~\cite{Dipierro-Savin-Valdinoci2020}) tendency of minimisers of the fractional perimeter to share a portion of boundary with the ambient domain~$\Omega$, effectively ``sticking'' to it.

In this paper, we investigate  stickiness properties of free boundary~$s$-minimal surfaces, see Section~\ref{sec: stickiness} below.
For this purpose, we give the following definitions: 
  
\begin{definition}[Stickiness]
	Let~$E$ and~$\Omega$ be open sets of~$\R^n$ 
	of class~$C^0$
and let~$x\in\partial E\cap\partial\Omega$. We say that 
the set~$E$ \emph{sticks to~$\Omega$} at~$x$ if there exists~$\rho>0$ such that
\begin{eqnarray*}&&
{\mbox{either $\Omega\cap B_\rho(x)\subset  E$ and $\Omega^c\cap B_\rho(x)\subset  E^c$}}\\&&
{\mbox{or $\Omega\cap B_\rho(x)\subset  E^c$ and $\Omega^c\cap B_\rho(x)\subset  E$.}}
\end{eqnarray*}
\end{definition}

\begin{definition}[Stickiness from outside]
	Let~$E$ and~$\Omega$ be open sets of~$\R^n$ of class~$C^0$ and let~$x\in\partial E\cap\partial\Omega$. We say that 
the set~$E$ \emph{sticks to~$\Omega$ from outside} at~$x$ if there exists~$\rho>0$ such that
\begin{eqnarray*}&&
{\mbox{either~$\Omega\cap B_\rho(x)\subset  E$ and~$\Omega^c\cap B_\rho(x)\cap E\ne\varnothing$}}\\&&
{\mbox{or~$\Omega\cap B_\rho(x)\subset  E^c$ and~$\Omega^c\cap B_\rho(x)\cap E^c\ne\varnothing$.}}\end{eqnarray*}
\end{definition}

\begin{definition}[Stickiness from inside]
	Let~$E$ and~$\Omega$ be open sets of~$\R^n$ of class~$C^0$ and let~$x\in\partial E\cap\partial\Omega$. We say that 
the set~$E$ \emph{sticks to~$\Omega$ from inside} at~$x$ if there exists~$\rho>0$ such that
\begin{eqnarray*}&&{\mbox{either~$\Omega^c\cap B_\rho(x)\subset  E^c$ and~$\Omega\cap B_\rho(x)\cap E^c\ne\varnothing$}}\\
&&{\mbox{or~$\Omega^c\cap B_\rho(x)\subset  E$ and~$\Omega\cap B_\rho(x)\cap E\ne\varnothing$.}}
	\end{eqnarray*}
\end{definition}

See Figure~\ref{fig: stickiness} for a depiction of the different types of stickiness.

\medskip

In the context of free boundary~$s$-minimal hypersurfaces, stickiness is quite subtle. Indeed, the behaviour that
it describes is exactly the opposite of what one expects from the local counterpart, where the hypersurface meets the boundary orthogonally.

On the other hand, at the level of the first variation~\eqref{eq: first variation}, in any portion of~$\partial E$ sticking to~$\partial\Omega$ the term~${\mathcal{X}}\cdot\nu_{\partial E}$ vanishes, as a consequence of~\eqref{eq: tangency}. This brings us to the following simple but instructive example.

\begin{example}[Total stickiness]\label{ex: tot stick} The set~$\Omega$ itself, and its complement, are free boundary~$s$-minimal surfaces in~$\Omega$ for any~$s\in(0,1)$. 

Indeed, by the tangency condition~\eqref{eq: tangency}, the flow leaves~$\Omega$ unchanged, and therefore $\Per_s(\Omega;\Omega)$ is constant for inner variations preserving~$\partial\Omega$.  
\end{example}

We point out that Example~\ref{ex: tot stick}
is a degenerate situation, which is the nonlocal analogue of the fact that (if one allows it) $\partial \Omega$ is a classical
free boundary minimal hypersurface in~$\Omega$, since it is unchanged by flows tangent to~$\partial \Omega$. 

Next, we show that there are non-trivial examples of stickiness for critical points.

\begin{example}\label{ex: sections}
	Let~$\Omega\coloneqq B_1\subset\R^n$ and 
	\begin{equation*}
		E\coloneqq\{x\in B_1:x_n>0\}\cup \{x\in B_1^c:x_n<0\}.
	\end{equation*}	
	Then, $E$ is a free boundary~$s$-minimal surface in~$\Omega$ for every $s\in(0,1)$, with
	\begin{equation*}
	\Lambda\coloneqq \left\lbrace x\in\partial E\cap \partial\Omega : \text{$E$ sticks to~$\Omega$ at~$x$}\right\rbrace\ne\varnothing
\end{equation*}
(in fact, with~$\Lambda$ dense in~$\partial\Omega$, being the unit sphere minus the equator). Indeed, in all points of~$\partial E\cap\Omega$ and~$\partial E\cap \overline\Omega^c$ the equations~\eqref{eq: s-min surf} and~\eqref{eq: free boundary} are strongly satisfied by symmetry. 
	
	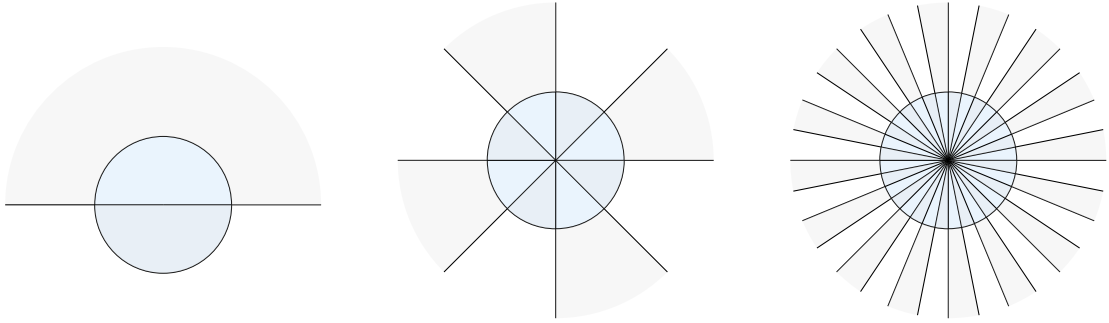
\begin{figure}
  \centering
  \foreach \john in {1,3,5} {
    \begin{minipage}{0.30\textwidth}
      \centering
      \begin{tikzpicture}[scale=0.7]
	
	\def\R{3}
	\def\r{1.3}
	
	\filldraw[fill=\colorset!10, draw=black] (0,0) circle (\r);

	\def\n{\john} 
  \pgfmathsetmacro{\k}{int(2^\n)}   
  \pgfmathsetmacro{\angle}{360/\k}  

  \foreach \i in {0,...,\numexpr\k-1} {
    \pgfmathsetmacro{\start}{\i*\angle}
    \pgfmathsetmacro{\end}{(\i+1)*\angle}
    
    \draw (0,0) -- ++(\start:\R);

    \ifodd\i
      \fill[fill=gray!20, fill opacity=0.3] (0,0) -- ({\r*cos(\start)}, {\r*sin(\start)})
                          arc[start angle=\start, end angle=\end, radius=\r]
                          -- cycle;
    \else
     \fill[fill=gray!20, fill opacity=0.3]
      ({\r*cos(\end)}, {\r*sin(\end)})
        arc[start angle=\end, delta angle=\start-\end, radius=\r]
      -- ({\R*cos(\start)}, {\R*sin(\start)})
        arc[start angle=\start, delta angle=\end-\start, radius=\R]
      -- cycle;      
    \fi
  }

	\end{tikzpicture}
    \end{minipage}
    \hspace{0.02\textwidth}
  }
  \caption{The set~$E$ (in grey) is a free boundary~$s$-minimal surface, with stickiness, in the unit circle (in light blue).}
  \label{fig: sections}
\end{figure}
	
	Using the same strategy we can construct many other examples.
	For any integer~$k\geq 1$ consider the set 
	\begin{equation*}
		\widetilde E\coloneqq \left\lbrace re^{i\theta}\in\C\simeq\R^2:r>0,\ \theta\in \left(\frac{\pi}{k}(2j),\frac{\pi}{k}(2j+1)\right),\ j=0,\dots,k-1\right\rbrace
	\end{equation*}
	given by cone in~$\R^2$ over every other arc in~$\s^1$ connecting~$2k$ equi-spaced points.
	
	One can check easily, still by symmetry, that~$\partial \widetilde E$ is a (non sticking) free boundary minimal hypersurface in~$B_1$. 
	
	On the other hand, the set 
	\begin{equation*}
		E\coloneqq (\widetilde E\cap B_1)\cup (\widetilde E^c\cap B_1^c)
	\end{equation*}
	is a free boundary~$s$-minimal hypersurface in~$B_1$, sticking on the whole unit circle minus the~$2k$ points, see Figure~\ref{fig: sections}.
\end{example}

The sets constructed in Example~\ref{ex: sections} show that stickiness can happen for free boundary~$s$-minimal hypersurfaces. Our next result establishes that the same is not true if we consider only stickiness from outside.

\begin{theorem}\label{dj9asofhvkrew8yhmoijyuj}
	Let~$s\in(0,1)$ and~$E$ be a free boundary~$s$-minimal surface in~$\Omega$ of class~$C^0$.
	
	Then, there are no points on~$\partial E\cap \partial\Omega$ at which~$E$ sticks to~$\Omega$ from outside.
\end{theorem}

We remark that there is no counterpart of Theorem~\ref{dj9asofhvkrew8yhmoijyuj} in the setting of sticking from inside.
Indeed, sticking from inside can happen in general. A simple example is given by
\begin{eqnarray*}E&\coloneqq &\{(y',y_n)\in\R^n:y_n>0\}\\ {\mbox{and }}\quad
\Omega&\coloneqq&\{(y',y_n)\in\R^n:y_n>\varphi(y')\},\end{eqnarray*} where~$\varphi$ is a smooth function such that~$\varphi\equiv 0$ in~$B_1$ and~$\varphi<0$ in~$B_1^c$. 

This kind of ``lack of symmetry'' for the two types of stickiness, from outside and from inside, is due to the fact the equations in~\eqref{eq: s-min surf}
and~\eqref{eq: free boundary} are fundamentally different. Indeed, the proof of Theorem~\ref{dj9asofhvkrew8yhmoijyuj}
exploits a limiting process for the free boundary equation that allows us to reach a contradiction, while  
a similar limiting process for the mean curvature equation would be inconclusive since there could be a compensation between the contributions from inside and outside of~$\Omega$. 
	
	It would be interesting to investigate the phenomenon of sticking from inside
in cases in which~$\Omega$ is convex, mean-convex or even just bounded.
\medskip

Nevertheless we are able to show that, if~$\partial E$ crosses~$\partial\Omega$ without sticking and it is regular enough away from~$\partial\Omega$, then the intersection is orthogonal and the set is regular across~$\partial\Omega$.
	
\begin{theorem}\label{R-wf0jgvhr90tiuohjtktyu}
	Let~$s\in(0,1)$ and~$\alpha\in(s,1]$.
	Let~$\Omega$ be an open set of~$\R^n$ of class~$C^{1,\alpha}$.
	Let~$E$ be a free boundary~$s$-minimal surface in~$\Omega$.
	
Assume that~$0\in\partial E\cap\partial \Omega$ and that there exist~$r>0$ and
	a diffeomorphism~$T:\R^n\to\R^n$ of class~$C^{1,\alpha}$,
with~$T(0)=0$, $ DT(0)=\operatorname{Id}$,
$T(B_r)=B_r$, $T(B_r\cap\Omega)=B_r\cap\{x_1>0\}$
and~$T(E\cap B_r)=E_1\cup E_2$, where
\begin{equation}\label{oqwdlfrthp:wedf}
\begin{split}& E_1\coloneqq \big\{ {\mbox{$x\in B_r$ s.t. $x_1\ge0$ and $\omega_{1}\cdot x<0$}}\big\}\\{\mbox{and }}  \;&
E_2\coloneqq \big\{ {\mbox{$x\in B_r$ s.t. $x_1\le0$ and $\omega_{2}\cdot x<0$}}\big\},
\end{split}\end{equation}
for some unit vectors~$\omega_{1}$, $\omega_{2}\in\R^n$ of the form
$$\omega_{1}=(-\sin\vartheta_1,0\dots,0,\cos\vartheta_1)\quad{\mbox{ and }}\quad\omega_{2}=(-\sin\vartheta_2,0\dots,0,\cos\vartheta_2).$$

Assume also that \begin{equation}\label{qew0doeflj234rt39ht6:gb}
\vartheta_1,\vartheta_2\in\left(-\frac\pi2,\frac\pi2\right).\end{equation}
Then $\vartheta_1=\vartheta_2=0$, and in particular ~$E$ is of class~$C^{1,\alpha}$ in the vicinity of the origin
	and the intersection of~$\partial E$ and~$\partial\Omega$ at the origin is orthogonal. 
\end{theorem}

We stress that condition~\eqref{qew0doeflj234rt39ht6:gb} is related to the fact that~$\partial E$ and~$\partial\Omega$ do not adhere to each other at the origin (which is supposed to be a common boundary point).
Moreover, condition~\eqref{oqwdlfrthp:wedf} says, in a nutshell, that~$E$
is of class~$C^{1,\alpha}$ near the origin ``from both sides'' of~$\partial\Omega$.
In this spirit, Theorem~\ref{R-wf0jgvhr90tiuohjtktyu} guarantees
that~$E$ is, in fact, $C^{1,\alpha}$ through the origin as well.

Natural questions regarding the nonlocal free boundary minimal surfaces involve their regularity and density properties. For instance one may wonder whether they are always smooth, whether a point on their boundary presents uniform densities of the set and its complement, and whether, in a given ball, it always presents two balls of comparable radii contained, respectively, in the set and its complement (this is the so-called ``clean ball condition''). Interestingly, all these properties do
{\em not} hold true in our setting. As a counterexample, one can consider, for all~$N\in{\mathbb{N}}\cap[2,+\infty)$, the planar set defined in complex notation by
$$E_N\coloneqq \left\{ re^{i\vartheta}, {\mbox{with~$r>0$ and }}\vartheta\in \bigcup_{j=0}^{N-1}\left(\frac{2j\pi}{N},\frac{(2j+1)\pi}{N}\right) \right\}.$$
By symmetry, $E_N$ is a nonlocal free boundary minimal surface in any ball centered at the origin, but, at the origin, it violates smoothness, and does not satisfy density estimates and clean ball conditions with respect to uniform quantities (any property of this type actually degenerates when~$N$ gets larger and larger).
\medskip

We now dive into the technical part of this paper, by providing the proofs of the main results.

\section{The Euler--Lagrange equations and proof of Theorem~\ref{PROP1}}\label{sec: EL}

In this section we prove Theorem~\ref{PROP1}.
We will follow the strategy of~\cite[Theorem~6.1]{Figalli-Fusco-Maggi-Millot-Morini2015}, by first regularising the kernel and then showing that the result passes to the limit. We begin by proving Lemma \ref{lem: estimate K near Omega}.

\begin{proof}[Proof of Lemma \ref{lem: estimate K near Omega}]
For every~$x\in T_{\bar\delta}$, we denote by~$\pi_x\in\partial\Omega$ the nearest point projection of~$x$ and let~$d_x\coloneqq \operatorname{dist}(x,\Omega)=|x-\pi_x|$.

Now let~$x \in T_{\bar\delta}\cap\overline\Omega^c$ and let~$P_x\colon \R^n\to\R^n$ be a rigid motion that maps~$\pi_x$ to~$0$ and~$x$ to~$d_xe_n$. Then, recalling~\eqref{nvcmxwkdefaddmiss},
and using the changes of variables~$z\coloneqq P_x y$ and~$w\coloneqq z/d_x$,
we find that
		\begin{equation}\label{vbncmxow8459rtegfekuw}\begin{split}
		&	\int_\Omega K(x-y)dy\leq C_K\int_\Omega \frac{dy}{|x-y|^{n+s}}\\
&\qquad=C_K\int_{P_x\Omega} \frac{dz}{|d_xe_n-z|^{n+s}}=\frac{C_K}{d_x^s}\int_{d_x^{-1}P_x\Omega}\frac{dw}{|e_n-w|^{n+s}}
.		\end{split}\end{equation}

We point out that~$d_x^{-1}P_x\Omega$ converges to~$\{y_n<0\}$ as~$d_x\to0$, namely as~$x\to \partial\Omega$.
As a result, there exists~$\delta_0\in(0,\bar\delta)$ such that,
for all~$x\in T_{\delta_0}\cap\overline\Omega^c$,
$$ \int_{d_x^{-1}P_x\Omega}\frac{dw}{|e_n-w|^{n+s}}\le 
1+\int_{\{w_n<0\}}\frac{dw}{|e_n-w|^{n+s}}<+\infty.$$

Plugging this information into~\eqref{vbncmxow8459rtegfekuw},
we conclude that, for all~$x\in T_{\delta_0}\cap\overline\Omega^c$,
$$\int_\Omega K(x-y)dy\le
\frac{C\,C_K}{d_x^s},$$
for some~$C>0$ depending only on~$n$ and~$s$.
This entails the desired result.
	\end{proof}

\begin{proof}[Proof of Theorem~\ref{PROP1}]
Given~$\delta>0$ sufficiently small, consider a smooth, monotone family of cut-off functions~$\eta_\delta\colon [0,+\infty)\to [0,1]$ such that
\begin{eqnarray*} \eta_\delta\equiv 1 {\mbox{ in }}[0,\delta]\cup \left[\tfrac1{\delta},+\infty\right),\quad 
\eta_\delta\equiv 0 {\mbox{ in }}\left[2\delta,\tfrac{1}{2\delta}\right]\quad{\mbox{ and }}\quad
|\eta'_\delta|\leq \tfrac{2}\delta.\end{eqnarray*}
We define the regularised kernel~$K_\delta(z)\coloneqq (1-\eta_\delta(|z|))K(z)$.

Note that, as a consequence of the tangency condition~\eqref{eq: tangency}, the flow leaves~$\Omega$ unchanged, namely
\begin{equation*}
	\Phi_t(\Omega)=\Omega\quad (\text{and clearly }\Phi_t(\Omega^c)=\Omega^c).
\end{equation*}
Moreover, we have that
$$ \frac{d}{dt}\Per_{K_\delta}(E_t;\Omega)= \frac{d}{dh}\Big\vert_{h=0}\Per_{K_\delta}(E_{t+h};\Omega),$$
see page~481 in~\cite{Figalli-Fusco-Maggi-Millot-Morini2015}.

Using these facts, we compute, for~$|t|$ and~$|h|$ small,
\begin{eqnarray*}
	\mathcal{I}_{K_\delta}(E_{t+h}^c\cap \Omega,E_{t+h}\cap\Omega)&=&\int_{E_{t+h}^c\cap \Omega}\int_{E_{t+h}\cap\Omega}K_\delta(x-y)dxdy
	\\&=&\int_{E_t^c\cap \Omega}\int_{E_t\cap\Omega}K_\delta(\Phi_h(x)-\Phi_h(y))J_{\Phi_h}(x)J_{\Phi_h}(y)dxdy.
\end{eqnarray*}
where~$J_{\Phi_h}$ is the Jacobian determinant of the change of variable~$\Phi_h$, which can be expanded as~$J_{\Phi_h}=\operatorname{id}+h{\mathcal{X}}+O(h^2)$. 

Thus, we can write
\begin{eqnarray*}&&
	\frac{d}{dt} \mathcal{I}_{K_\delta}(E_t^c\cap \Omega,E_t\cap\Omega)=A_1+A_2,
\end{eqnarray*}
where
\begin{eqnarray*}
A_1&\coloneqq &\int_{E_t^c\cap \Omega}\int_{E_t\cap\Omega}\nabla K_\delta(x-y)({\mathcal{X}}(x)-{\mathcal{X}}(y))dxdy\\
{\mbox{and }}\quad A_2&\coloneqq &\int_{E_t^c\cap \Omega}\int_{E_t\cap\Omega}K_\delta(x-y)(\div {\mathcal{X}}(x)+\div {\mathcal{X}}(y))dxdy.\end{eqnarray*}

By the symmetry of~$K$ and integrating by parts in~$A_1$, we get 
\begin{align*}
	A_1&=\int_{E_t^c\cap \Omega}\left(\int_{E_t\cap\Omega}\nabla K_\delta(x-y)\cdot{\mathcal{X}}(x)dx\right)dy+\int_{E_t\cap \Omega}\left(\int_{E_t^c\cap\Omega}\nabla K_\delta(y-x)\cdot{\mathcal{X}}(y)dy\right)dx\\
	&=-\int_{E_t^c\cap \Omega}\int_{E_t\cap\Omega} K_\delta(x-y)\div {\mathcal{X}}(x)dxdy\\
	&\qquad\qquad+\int_{E_t^c\cap \Omega}\int_{\partial (E_t\cap\Omega)}K_\delta(x-y){\mathcal{X}}(x)\cdot\nu_{\partial (E_t\cap\Omega)}(x)d\mathscr{H}^{n-1}_xdy\\
	&\qquad\qquad-\int_{E_t\cap \Omega}\int_{E_t^c\cap\Omega} K_\delta(x-y)\div {\mathcal{X}}(y)dxdy\\
	&\qquad\qquad+\int_{E_t\cap \Omega}\int_{\partial (E_t^c\cap\Omega)}K_\delta(x-y){\mathcal{X}}(y)\cdot\nu_{\partial (E_t^c\cap\Omega)}(y)d\mathscr{H}^{n-1}_ydx.
\end{align*}
We observe that the sum of the first and the third terms in the last expression equals~$-A_2$, and therefore
\begin{equation}\label{yfhjdkls874395820yhfjd}\begin{split}&
	\frac{d}{dt} \mathcal{I}_{K_\delta}(E_t^c\cap \Omega,E_t\cap\Omega)\\&\quad=
	\int_{E_t^c\cap \Omega}\int_{\partial (E_t\cap\Omega)}K_\delta(x-y){\mathcal{X}}(x)\cdot\nu_{\partial (E_t\cap\Omega)}(x)d\mathscr{H}^{n-1}_xdy
	\\&\qquad\quad+\int_{E_t\cap \Omega}\int_{\partial (E_t^c\cap\Omega)}K_\delta(x-y){\mathcal{X}}(y)\cdot\nu_{\partial (E_t^c\cap\Omega)}(y)d\mathscr{H}^{n-1}_ydx.
\end{split}\end{equation}

We also remark that~$\partial (E_t\cap\Omega)=(\partial E_t\cap \Omega)\cup(E_t\cap\partial\Omega)$ and, for~$x\in E_t\cap\partial\Omega$,
\begin{equation*}
	\nu_{\partial (E_t\cap\Omega)}(x)=\nu_{\partial\Omega}(x).
\end{equation*}
Hence, by the tangency condition~\eqref{eq: tangency}, we deduce that
\begin{equation*}
	{\mathcal{X}}(x)\cdot \nu_{\partial (E_t\cap\Omega)}(x)=0\quad {\mbox{ for all }} x\in E_t\cap\partial\Omega.
\end{equation*}

Moreover,
\begin{eqnarray*}\partial (E_t^c\cap\Omega)=(\partial E_t^c\cap \Omega)\cup(E_t^c\cap\partial\Omega)=
(\partial E_t\cap \Omega)\cup(E_t^c\cap\partial\Omega)
\end{eqnarray*}
and we have that
\begin{eqnarray*}
&&{\mathcal{X}}(y)\cdot \nu_{\partial (E_t^c\cap\Omega)}(y)=0\quad {\mbox{ for all }} y\in E_t^c\cap\partial\Omega\\
{\mbox{and }} &&
{\mathcal{X}}(y)\cdot \nu_{\partial (E_t^c\cap\Omega)}(y)=-{\mathcal{X}}(y)\cdot \nu_{\partial (E_t\cap\Omega)}(y)
\quad {\mbox{ for all }} y\in \partial E_t^c \cap\Omega.
\end{eqnarray*}

As a result, using these pieces of information into~\eqref{yfhjdkls874395820yhfjd}, we conclude that
\begin{eqnarray*}&&
	\frac{d}{dt} \mathcal{I}_{K_\delta}(E_t^c\cap \Omega,E_t\cap\Omega)\\&&\quad
	=\int_{E_t^c\cap \Omega}\int_{\partial E_t\cap\Omega}K_\delta(x-y){\mathcal{X}}(x)\cdot\nu_{\partial E_t}(x)d\mathscr{H}^{n-1}_xdy\\
	&&\qquad\quad-\int_{E_t\cap \Omega}\int_{\partial E_t\cap\Omega}K_\delta(x-y){\mathcal{X}}(y)\cdot\nu_{\partial E_t}(y)d\mathscr{H}^{n-1}_ydx.
\end{eqnarray*}

Similar computations lead to 
\begin{eqnarray*}&&
	\frac{d}{dt} \mathcal{I}_{K_\delta}(E_t\cap \Omega,E_t^c\cap\Omega^c)\\&&\quad=\int_{E_t^c\cap \Omega^c}\int_{\partial E_t\cap\Omega}K_\delta(x-y){\mathcal{X}}(x)\cdot\nu_{\partial E_t}(x)d\mathscr{H}^{n-1}_xdy\\
	&&\qquad\quad-\int_{E_t\cap \Omega}\int_{\partial E_t\cap\Omega^c}K_\delta(x-y){\mathcal{X}}(y)\cdot\nu_{\partial E_t}(y)d\mathscr{H}^{n-1}_ydx
\end{eqnarray*}
and 
\begin{eqnarray*}&&
	\frac{d}{dt} \mathcal{I}_{K_\delta}(E_t\cap \Omega^c,E_t^c\cap\Omega)\\&&\quad=\int_{E_t^c\cap \Omega}\int_{\partial E_t\cap\Omega^c}K_\delta(x-y){\mathcal{X}}(x)\cdot\nu_{\partial E_t}(x)d\mathscr{H}^{n-1}_xdy\\
	&&\qquad\quad-\int_{E_t\cap \Omega^c}\int_{\partial E_t\cap\Omega}K_\delta(x-y){\mathcal{X}}(y)\cdot\nu_{\partial E_t}(y)d\mathscr{H}^{n-1}_ydx.
\end{eqnarray*}
Thus, putting everything together, we find that, for~$|t|$ small,
\begin{equation}\label{compuaboveujfyrjfn86954}\begin{split}&
	\frac{d}{dt} \Per_{K_\delta}(E_t;\Omega)\\&\quad=\int_{\partial E_t\cap \Omega}\left(\int_{E_t^c}K_\delta(x-y)dy-\int_{E_t}K_\delta(x-y)dy\right){\mathcal{X}}(x)\cdot\nu_{\partial E_t}(x)d\mathscr{H}^{n-1}_x\\
&\qquad+\int_{\partial E_t\cap\Omega^c}\left(\int_{E_t^c\cap \Omega}K_\delta(x-y)dy-\int_{E_t\cap \Omega}K_\delta(x-y)dy\right){\mathcal{X}}(x)\cdot\nu_{\partial E_t}(x)d\mathscr{H}^{n-1}_x.
\end{split}\end{equation}

Next, we consider the limit as~$\delta\searrow 0^+$. For this, we set 
\begin{equation*}
	\phi_\delta(t)\coloneqq \Per_{K_\delta}(E_t;\Omega)\qquad\text{and}\qquad \phi(t)\coloneqq \Per_{K}(E_t;\Omega).
\end{equation*}
Our goal is to show that 
\begin{equation}\label{eq: first variation K}
	\phi'(0)=\int_{\partial E\cap \Omega}\H^K_{E}(x)\xi(x)d\mathscr{H}^{n-1}_x+\int_{\partial E\cap \Omega^c}\A^{K}_{E,\Omega}(x)\xi(x)d\mathscr{H}^{n-1}_x
\end{equation} 
where~$\xi(x)\coloneqq {\mathcal{X}}(x)\cdot \nu_{\partial E}(x)$.

To this aim, we notice that, by the Monotone Convergence Theorem, for~$|t|$ small, say~$|t|<\varepsilon$,
\begin{equation}\label{7985430fhcldsigehilehrjhgfd654ligerw}
\lim_{\delta\searrow0}\phi_\delta(t)=\phi(t).\end{equation} 
Moreover, by~\eqref{compuaboveujfyrjfn86954},
\begin{equation}\label{ytiuerw9865tyfgkjds}
	\phi_\delta'(t)=\int_{\partial E_t\cap \Omega}\H^{K_\delta}_{E_t}(x)\xi(x)d\mathscr{H}^{n-1}_x+\int_{\partial E_t\cap \Omega^c}\A^{K_\delta}_{E_t,\Omega}(x)\xi(x)d\mathscr{H}^{n-1}_x.
\end{equation}

Next, we perform a passage to the limit as~$\delta\searrow 0$ of the derivative by showing uniform convergence to the desired quantity.
In this step, the argument to show convergence of~$\H_{E_t}^{K_\delta}$ and~$\A_{E_t,\Omega}^{K_\delta}$ is very different.
Indeed, by~\cite[Proposition~6.3]{Figalli-Fusco-Maggi-Millot-Morini2015} the approximated curvatures $\H_{E_t}^{K_\delta}$ are uniformly close to~$\H_{E_t}^{K}$ in the whole region~$\partial E_t\cap \Omega$. Note that the proof of~\cite[Proposition~6.3]{Figalli-Fusco-Maggi-Millot-Morini2015} provides estimates in $\Omega$ which are robust as long as $E$ is $C^{1,\alpha}$ in $\Omega$, with~$\alpha\in(s,1)$. In particular, we have that 
\begin{equation}\label{ytiuerw9865tyfgkjds2}
	\lim_{\delta\to 0^+}\;\sup_{|t|<\varepsilon}\; \sup_{\partial E_t\cap\Omega\cap\operatorname{spt}\xi}\big|\H^{K}_{E_t}-\H^{K_\delta}_{E_t}\big|=0,
\end{equation}
see~\cite[formula~(6.24)]{Figalli-Fusco-Maggi-Millot-Morini2015}.

The same argument cannot be made for $\A_{E_t,\Omega}^{K_\delta}$, since it would imply uniform convergence of the bounded quantities $\A_{E_t,\Omega}^{K_\delta}$ to the unbounded quantity $\A_{E_t,\Omega}^{K}$ in $\partial E\cap\Omega^c$. Instead, we proceed as follows: let~$\delta_0$ be given by Lemma~\ref{lem: estimate K near Omega} and let~$\delta\in\left(0,\frac{\delta_0}2\right)$.
Let~$T_{2\delta}$ be a~${2\delta}$-tubular neighbourhood of~$\partial\Omega$. Note that, for all~$x\in \partial E_t\cap\overline\Omega^c\cap T_{2\delta}^c$,
\begin{equation*}
	\begin{split}&
		\big|\A_{E_t,\Omega}^{K_{{\delta}}}(x)-\A_{E_t,\Omega}^{K}(x)\big|=\left|\int_\Omega(\chi_{E_t^c}(y)-\chi_{E_t}(y))\eta_{\delta}(|x-y|)K(x-y)dy \right|\\
	&\qquad\leq 2\int_{B_{1/2\delta}^c}K(z)dz
	\leq 2C_K\int_{B_{1/2\delta}^c}\frac{dz}{|z|^{n+s}}	
	= \frac{2^{1+s}C_K\omega_n}{s}\,\delta^s.
	\end{split}
\end{equation*}
As a consequence, since~$\xi$ is compactly supported,
\begin{equation}\label{vcxnwqyet624ertdfgcg78iuhj}\begin{split} &
\left|\int_{\partial E_t\cap \Omega^c\cap T_{2\delta}^c}\Big(\A^{K_\delta}_{E_t,\Omega}(x)-\A^{K}_{E_t,\Omega}(x)\Big)\xi(x)d\mathscr{H}^{n-1}_x\right|
\le C\|{\mathcal{X}}\|_{L^\infty(\R^n,\R^n)} \delta^s
,\end{split}\end{equation}
up to renaming~$C$.

Furthermore, if~$x\in \partial E_t\cap\overline\Omega^c\cap T_{2\delta}$, thanks to Lemma \ref{lem: estimate K near Omega}
we have that
\begin{equation*}
\begin{split}
	&\big|\A_{E_t,\Omega}^{K_\delta}(x)-\A_{E_t,\Omega}^{K}(x)\big|=\left|\int_\Omega(\chi_{E_t^c}(y)-\chi_{E_t}(y))\eta_\delta(|x-y|)K(x-y)dy \right|\\
	&\qquad\qquad\leq \int_\Omega K(x-y)dy\leq C\operatorname{dist}(x,\Omega)^{-s},
\end{split}
\end{equation*}
where $C>0$ depends on~$C_K$, $n$ and~$s$.

Therefore,
\begin{equation}\label{y5t43wTUfrgcnxty54}\begin{split}&
\left|\int_{\partial E_t\cap \Omega^c\cap T_{2\delta}}\Big(
\A^{K_\delta}_{E_t,\Omega}(x)-\A^{K}_{E_t,\Omega}(x)\Big)\xi(x)d\mathscr{H}^{n-1}_x\right|\\&\qquad\le C \int_{\partial E_t\cap \Omega^c\cap T_{2\delta}}
\operatorname{dist}(x,\Omega)^{-s}
|\xi(x)|d\mathscr{H}^{n-1}_x.
\end{split}\end{equation}

Now, we suppose that the support of~${\mathcal{X}}$ is contained in some ball~$B_R$, and we use the regularity assumption on~$\partial E$ to conclude that there exist~$\rho>0$, $N\in{\mathbb{N}}$ and~$p_1,\dots,p_N\in\R^n$ such that, for~$t$ sufficiently small,
$$
\partial E_t\cap B_R\subset\bigcup_{i=1}^N B_\rho(p_i)$$
and $\partial E_t\cap B_\rho(p_i)$ is the graph of a~$C^{1,\alpha}$-function.

Using this information into~\eqref{y5t43wTUfrgcnxty54}, we find that
\begin{equation}\label{vckm87654eworu948}\begin{split}&
\left|\int_{\partial E_t\cap \Omega^c\cap T_{2\delta}}\Big(
\A^{K_\delta}_{E_t,\Omega}(x)-\A^{K}_{E_t,\Omega}(x)\Big)\xi(x)d\mathscr{H}^{n-1}_x\right|\\&\qquad\le C \int_{\partial E_t\cap \Omega^c\cap T_{2\delta}\cap B_R}
\operatorname{dist}(x,\Omega)^{-s}
|\xi(x)|d\mathscr{H}^{n-1}_x\\
&\qquad\le C\sum_{i=1}^N \int_{\partial E_t\cap \Omega^c\cap T_{2\delta}\cap B_{\rho}(p_i)}\operatorname{dist}(x,\Omega)^{-s}
|\xi(x)|d\mathscr{H}^{n-1}_x\\&\qquad\le
C\|{\mathcal{X}}\|_{L^\infty(\R^n,\R^n)}
\sum_{i=1}^N \int_{\partial E_t\cap \Omega^c\cap T_{2\delta}\cap B_{\rho}(p_i)}\operatorname{dist}(x,\Omega)^{-s}d\mathscr{H}^{n-1}_x
.\end{split}\end{equation}

Now we let~$\pi_x\in\partial\Omega$ be such that~$|x-\pi_x|=\operatorname{dist}(x,\Omega)$ and consider a diffeomorphism of~$B_{\rho}(p_i)$ of class~$C^1$ which places~$\partial E_t$ into~$\Sigma\coloneqq \{x_n=0\}$ and~$\Omega$ 
into~$L\coloneqq \{\cos\theta\,x_n<\sin\theta\, x_{n-1}\}$ for some~$\theta\in\left(0,\frac\pi2\right]$ (we stress that we are using here the transversality assumption between~$E_t$ and~$\Omega$).

In this way,
\begin{equation}\label{KLasd456y}\begin{split}
&\int_{\partial E_t\cap \Omega^c\cap T_{2\delta}\cap B_{\rho}(p_i)}\operatorname{dist}(x,\Omega)^{-s}d\mathscr{H}^{n-1}_x=
\int_{\partial E_t\cap \Omega^c\cap T_{2\delta}\cap B_{\rho}(p_i)}\frac{d\mathscr{H}^{n-1}_x}{|x-\pi_x|^s}
\\&\qquad \le C
\int_{\Sigma\cap \{x_{n-1}\in(0,4\delta)\}
\cap \{ |(x_1,\dots,x_{n-2})|<2\rho\} }
\frac{dx_1\dots dx_{n-1}}{|x-\Pi_x|^s},
\end{split}
\end{equation}
for a suitable~$\Pi_x\in\partial L$ (which is the image of the old projection~$\pi_x$ under this diffeomorphism).

We now observe that the distance of~$x=(x_1,\dots,x_{n-1},0)$ to~$\partial L$ is~$\sin\theta\,x_{n-1}$ and therefore~$|x-\Pi_x|\ge \sin\theta\,x_{n-1}$.

Plugging this information into~\eqref{KLasd456y}, it follows that
\begin{eqnarray*}&&
\int_{\partial E_t\cap \Omega^c\cap T_{2\delta}\cap B_{\rho}(p_i)}\operatorname{dist}(x,\Omega)^{-s}d\mathscr{H}^{n-1}_x\\&&\qquad
\le\frac{C}{\sin^s\theta} \int_{\Sigma\cap \{x_{n-1}\in(0,4\delta)\}\cap \{ |(x_1,\dots,x_{n-2})|<2\rho\}}\frac{dx_1\dots dx_{n-1}}{x_{n-1}^s}\le \frac{C\rho^{n-2}\delta^{1-s}}{\sin^s\theta},\end{eqnarray*}
up to conveniently renaming~$C$.

{F}rom this and~\eqref{vckm87654eworu948}, we thus obtain that
\begin{eqnarray*}&&
\left|\int_{\partial E_t\cap \Omega^c\cap T_{2\delta}}\Big(
\A^{K_\delta}_{E_t,\Omega}(x)-\A^{K}_{E_t,\Omega}(x)\Big)\xi(x)d\mathscr{H}^{n-1}_x\right|
\le\frac{C\|{\mathcal{X}}\|_{L^\infty(\R^n,\R^n)}\rho^{n-2}\delta^{1-s}}{\sin^s\theta}.
\end{eqnarray*}
Using this and~\eqref{vcxnwqyet624ertdfgcg78iuhj}, we conclude that
\begin{align*}
&	\left|\int_{\partial E_t\cap \Omega^c}\Big(
\A^{K_\delta}_{E_t,\Omega}(x)-\A^{K}_{E_t,\Omega}(x)\Big)\xi(x)d\mathscr{H}^{n-1}_x\right|\\&\qquad\le
\left|\int_{\partial E_t\cap \Omega^c\cap T_{2\delta}}\Big(
\A^{K_\delta}_{E_t,\Omega}(x)-\A^{K}_{E_t,\Omega}(x)\Big)\xi(x)d\mathscr{H}^{n-1}_x\right|\\&\qquad\qquad
+\left|\int_{\partial E_t\cap \Omega^c\cap T_{2\delta}^c}\Big(\A^{K_\delta}_{E_t,\Omega}(x)-\A^{K}_{E_t,\Omega}(x)\Big)\xi(x)d\mathscr{H}^{n-1}_x\right|\\&\qquad\le C\|{\mathcal{X}}\|_{L^\infty(\R^n,\R^n)}\big(\delta^{1-s}+\delta^s\big),
\end{align*} up to relabelling~$C$.

Consequently,
\begin{align*}
	\lim_{\delta\searrow 0}\;\sup_{|t|<\varepsilon}\left|
	\int_{\partial E_t\cap \Omega^c}\Big(\A^{K_\delta}_{E_t,\Omega}(x)-\A^{K}_{E_t,\Omega}(x)\Big)\xi(x)d\mathscr{H}^{n-1}_x\right|\leq \lim_{\delta\searrow 0}\;\sup_{|t|<\varepsilon}C\|{\mathcal{X}}\|_\infty\delta^{\min\{s,1-s\}}=0.
\end{align*}
{F}rom this, \eqref{ytiuerw9865tyfgkjds} and~\eqref{ytiuerw9865tyfgkjds2}, we thus obtain that
$$  \lim_{\delta\to 0^+}\;\sup_{|t|<\varepsilon}\left|\phi_\delta'(t)
-\int_{\partial E_t\cap \Omega}\H^{K}_{E_t}(x)\xi(x)d\mathscr{H}^{n-1}_x-\int_{\partial E_t\cap \Omega^c}\A^{K}_{E_t,\Omega}(x)\xi(x)d\mathscr{H}^{n-1}_x\right|=0.$$

As a result, recalling also~\eqref{7985430fhcldsigehilehrjhgfd654ligerw} we conclude that, for~$|t|<\varepsilon$,
$$\phi'(t)=\int_{\partial E_t\cap \Omega}\H^K_{E_t}(x)\xi(x)d\mathscr{H}^{n-1}_x+\int_{\partial E_t\cap \Omega^c}\A^{K}_{E_t,\Omega}(x)\xi(x)d\mathscr{H}^{n-1}_x,$$
which gives the desired result in~\eqref{eq: first variation K} by taking~$t=0$.
\end{proof}
 
\section{The limit of the free boundary condition and proof of Theorem~\ref{prop: s to 1}}\label{sec: lim s1}

The proof of Theorem \ref{prop: s to 1} relies on a straightening procedure, based on two technical results which we state next. 

We will consider the following geometric setup: let $\Omega$ and $F$ be open sets of class $C^{1,1}$. 
Let~$T\colon \R^n\to\R^n$ be a diffeomorphism of~$\R^n$ of class~$C^{1,1}$
and suppose that~$T^{-1}(X)=X+S(X)$,
with~$S(0)=0$ and~$DS(0)=0$.

For $r\in(0,1)$, let $Q_r=(-r,r)^{n}$.
Let $U_r\coloneqq T^{-1}(Q_r)$ and assume that
\begin{align*}
	T(\Omega\cap U_r)&=
\{x_1<0\}\cap Q_r,\\
T(F\cap U_r)&=\{\omega\cdot x<0\}\cap Q_r
\end{align*}
with~$\omega=(-\sin\vartheta,0,\dots,0,\cos\vartheta)$, for some~$\vartheta\in\left(-\frac\pi2,\frac\pi2\right)$.

In the remainder of this section, we use uppercase letters $(X,Y,\dots)$ to denote points in the range of $T$ (that is, the straightened domain) and lowercase letters $(x,y,\dots)$ to denote points in the original domain. In particular, we use the notation~$X=T(x)$.

We denote 
\begin{align*}
	\mathscr{I}(X,Y)&\coloneqq \frac{X-Y}{|X-Y|}+\frac{DS(X)(X-Y)}{|X-Y|},\\
	\mathscr{W}(X,Y)&\coloneqq \frac{D^2S(X) [X-Y,X-Y]}{2|X-Y|},\\{\mbox{and }}\quad
	\mathscr{B}_s(X,Y)&\coloneqq -(n+s)|\mathscr{I}(X,Y)|^{-n-s-2}
\mathscr{I}(X,Y)\cdot\mathscr{W}(X,Y)
\end{align*}

\begin{lemma}\label{saxjoml3rf}
There exist~$r_0\in\left(0,\frac12\right)$ and~$C\ge1$, depending only on~$\Omega$, $F$, $\vartheta$, and~$n$, such that, if~$r\in(0,r_0)$,
for all~$x\in\partial F\cap \overline\Omega^c\cap U_{r}$,
\begin{equation*}
	\begin{split}
		\Bigg|
\int_{\Omega\cap U_r}\frac{\chi_F(y)}{|x-y|^{n+s}}dy&-\int_{\{Y_1<0\}\cap\{\omega\cdot Y<0\}\cap Q_r}|\mathscr{I}(X,Y)|^{-n-s}\frac{|\det DT^{-1}(Y)|}{|X-Y|^{n+s}}dY\\
	&-\int_{\{Y_1<0\}\cap\{\omega\cdot Y<0\}\cap Q_r}\mathscr{B}_s(X,Y)\frac{|\det DT^{-1}(Y)|}{|X-Y|^{n+s}}dY\Bigg|\\
&\leq \frac{Cr}{s}\left(X_1^{1-s}+\frac{(Cr)^{1-s}-X_1^{1-s}}{1-s}\right).
	\end{split}
\end{equation*}
\end{lemma}

\begin{proof}
	By a change of variable, we see that
\begin{equation}\label{URcw208f78M}\begin{split}
\int_{\Omega\cap U_r}\frac{\chi_F(y)}{|x-y|^{n+s}}\,dy&=
\int_{\{Y_1<0\}\cap\{\omega\cdot Y<0\}\cap Q_r}\frac{|\det DT^{-1}(Y)|}{|T^{-1}(X)-T^{-1}(Y)|^{n+s}}\,dY.
\end{split}
\end{equation}

We observe that, for $|X-Y|$ small,
\begin{equation}\label{03-44t-2wd}\begin{split}&
	|T^{-1}(X)-T^{-1}(Y)|^{-n-s}=|X-Y+S(X)-S(Y)|^{-n-s}\\&\quad
	=\left|X-Y+DS(X)(X-Y)-\frac{D^2S(X)}2[X-Y,X-Y]+O(|X-Y|^3)\right|^{-n-s}\\
	&\quad=|X-Y|^{-n-s}\left|\mathscr{I}(X,Y)+\mathscr{W}(X,Y)+O(|X-Y|^2)\right|^{-n-s}.
\end{split}
\end{equation}
Moreover, we have that, for small~$r$ and~$X\in Q_r$,
$$|\mathscr{I}(X,Y)|\ge\frac12.$$
Also, for $|X-Y|$ small,
\begin{align*}
\mathscr{W}(X,Y)=O(|X-Y|).
\end{align*}
Thus,
\begin{eqnarray*}
&&|\mathscr{I}(X,Y)+\mathscr{W}(X,Y)+O(|X-Y|^2)|^{-n-s}\\&&\quad=|\mathscr{I}(X,Y)|^{-n-s}-(n+s)|\mathscr{I}(X,Y)|^{-n-s-2}
\mathscr{I}(X,Y)\cdot\mathscr{W}(X,Y)+O(|X-Y|^{2}).
\end{eqnarray*}

Using this information into~\eqref{03-44t-2wd}, we obtain that
\begin{align}\label{Tr5p}
|T^{-1}(X)-T^{-1}(Y)|^{-n-s}=\frac{|\mathscr{I}(X,Y)|^{-n-s}}{|X-Y|^{n+s}}+\frac{\mathscr{B}_s(X,Y)}{|X-Y|^{n+s}}+O(|X-Y|^{2-n-s}).
\end{align}

Next, we set 
\begin{align*}
	\Xi(x)\coloneqq &\int_{\Omega\cap U_r}\frac{\chi_F(y)}{|x-y|^{n+s}}dy-\int_{\{Y_1<0\}\cap\{\omega\cdot Y<0\}\cap Q_r}|\mathscr{I}(X,Y)|^{-n-s}\frac{|\det DT^{-1}(Y)|}{|X-Y|^{n+s}}dY\\
	&\quad-\int_{\{Y_1<0\}\cap\{\omega\cdot Y<0\}\cap Q_r}\mathscr{B}_s(X,Y)\frac{|\det DT^{-1}(Y)|}{|X-Y|^{n+s}}dY.
\end{align*}
It follows from~\eqref{URcw208f78M} and~\eqref{Tr5p} that
\begin{equation}\label{I7DqidMNa34}
|\Xi(x)|\le C\int_{\{Y_1<0\}\cap\{\omega\cdot Y<0\}\cap Q_r}\frac{dY}{|X-Y|^{n+s-2}}.
\end{equation}

Now we use the notation~$Y=(Y_1,Y'',Y_n)\in\R\times\R^{n-2}\times\R$, $
\omega_0=(-\sin\vartheta,\cos\vartheta)
$, and~$\zeta\coloneqq (X_1,X_n)$,
and we find that
\begin{equation}\label{s0dvkp9ipEDCwV3}\begin{split}
&\int_{\{Y_1<0\}\cap\{\omega\cdot Y<0\}\cap Q_r}\frac{dY}{|X-Y|^{n+s-2}}\\&\quad=
\int_{\{Y_1<0\}\cap\{\omega\cdot Y<0\}\cap Q_r}\frac{dY}{\big(|(X_1,X_n)-(Y_1,Y_n)|^2+|X''-Y''|^2 \big)^{\frac{n+s-2}2}}\\&\quad\le 
\int_{{(\lambda,\mu)\in(-r,r)^2\times\R^{n-2}}\atop{
\{\lambda_1<0\}\cap\{\omega_0\cdot \lambda<0\}}}\frac{d\lambda\,d\mu}{\big(|\zeta-\lambda|^2+|\mu|^2 \big)^{\frac{n+s-2}2}}\\&\quad=
\int_{{(\lambda,\ell)\in(-r,r)^2\times\R^{n-2}}\atop{
\{\lambda_1<0\}\cap\{\omega_0\cdot \lambda<0\}}}\frac{d\lambda\,d\ell}{|\zeta-\lambda|^{s}\big(1+|\ell|^2 \big)^{\frac{n+s-2}2}}\\&\quad\le \frac{C}s
\int_{{\lambda\in(-r,r)^2}\atop{
\{\lambda_1<0\}\cap\{\omega_0\cdot \lambda<0\}}}\frac{d\lambda}{|\zeta-\lambda|^{s}}.
\end{split}\end{equation}

We also observe that~$X\in T(\partial F\cap U_r)$ and thus~$0=\omega\cdot X=\omega_0\cdot\zeta$,
yielding that\begin{equation*} \zeta_2=\zeta_1\tan\vartheta.\end{equation*}
Hence, if~$\varpi_0:=(\cos\vartheta,\sin\vartheta)$,
$$ \zeta\cdot\varpi_0=\zeta_1\cos\vartheta+\zeta_2\sin\vartheta=\frac{\zeta_1}{\cos\vartheta}.$$

Additionally, since~$X\in T(\Omega^c\cap U_r)$, we have that~$X_1\ge0$, thus~$\zeta_1\ge0$,
and consequently~$\frac{\zeta_1}{\cos\vartheta}\ge0$.
In particular, if~$\lambda_1<0$,
$$ \frac{2\zeta_1(\lambda\cdot\varpi_0)}{\cos\vartheta}
=2\zeta_1(\lambda_1+\lambda_2\tan\vartheta)\le2
\zeta_1\lambda_2\tan\vartheta\le \zeta_1^2 \tan^2\vartheta+\lambda_2^2,
$$
where the last step relies on the Cauchy--Schwarz inequality.

As a result, since
the vectors~$\omega_0$ and~$\varpi_0$ constitute an orthonormal basis of~$\R^2$, if~$\lambda_1<0$,
\begin{eqnarray*}&&
|\zeta-\lambda|^2=\big((\zeta-\lambda)\cdot\omega_0\big)^2+\big((\zeta-\lambda)\cdot\varpi_0\big)^2=
(\lambda\cdot\omega_0)^2+\left(\frac{\zeta_1}{\cos\vartheta}-\lambda\cdot\varpi_0\right)^2\\&&\qquad=
|\lambda|^2+\frac{\zeta_1^2}{\cos^2\vartheta}-\frac{2\zeta_1(\lambda\cdot\varpi_0)}{\cos\vartheta}\ge
|\lambda|^2+\frac{\zeta_1^2}{\cos^2\vartheta}-\zeta_1^2 \tan^2\vartheta-\lambda_2^2\\&&\qquad=\lambda_1^2+\zeta_1^2.
\end{eqnarray*}
Combining this and~\eqref{s0dvkp9ipEDCwV3}, we see that
\begin{equation*}\begin{split}
&\int_{\{Y_1<0\}\cap\{\omega\cdot Y<0\}\cap Q_r}\frac{dY}{|X-Y|^{n+s}}\le\frac{ Cr}s
\int_0^{r}\frac{d\lambda_1}{(\lambda_1^2+\zeta_1^2)^{\frac{s}2}}\\&\qquad\le \frac{Cr}s\left(
\int_{0}^{\zeta_1}\frac{d\lambda_1}{\zeta_1^{s}}+
\int_{\zeta_1}^r\frac{d\lambda_1}{\lambda_1^{s}}
\right)\le \frac{Cr}s\left(\zeta_1^{1-s}+\frac{r^{1-s}-\zeta_1^{1-s}}{1-s}\right).
\end{split}\end{equation*}

This, in tandem with~\eqref{I7DqidMNa34}, returns that
\begin{equation*}
|\Xi(x)|\le \frac{Cr}s\left(\zeta_1^{1-s}+\frac{r^{1-s}-\zeta_1^{1-s}}{1-s}\right),
\end{equation*}
as desired.
\end{proof}

\begin{lemma}\label{lem:ehfwbuebd}
Let $E$, $\Omega$, ${\mathcal{X}}$ and~$g$ be as in Theorem~\ref{prop: s to 1}. 

Then, there exists~$r_0\in\left(0,\frac12\right)$, depending only on~$\Omega$, $E$, and~$n$, such that, if~$r\in(0,r_0)$,
	\begin{eqnarray*}
	&&	\lim_{s\nearrow1}
\int_{\partial E\cap \Omega^c\cap U_r}\mathcal{A}^s_{E,\Omega}(x) {\mathcal{X}}(x)\cdot \nu_{\partial E}(x)\,d\mathscr{H}^{n-1}_x\\&&\qquad=\int_{\partial E\cap \partial \Omega\cap U_r}g(\psi_{x'})\mathcal{X}(x')\cdot\nu_{\partial E}(x')\,d{\mathscr{H}}^{n-2}_{x'}+O(r^{n-1}),
	\end{eqnarray*}
	where $\psi_{x'}$ is the intersection angle between the affine hyperplanes~$T_{x'}(\partial E)$ and~$T_{x'}(\partial \Omega)$.
\end{lemma}

\begin{proof}
We point out that, if~$x\in \partial E\cap U_{r}$,
\begin{align*}
	|\A^s_{E,\Omega\setminus U_r}(x)|&\le Cs(1-s)\left|
\int_{\Omega\setminus U_r}\frac{\chi_{E^c}(y)-\chi_{E}(y)}{|x-y|^{n+s}}\,dy\right|\\
&\le
Cs(1-s)\int_{\R^n\setminus B_{r}}\frac{dz}{|z|^{n+s}}\\
&\le \frac{C(1-s)}{r^s}.
\end{align*}
Therefore, by the Dominated Convergence Theorem,
\begin{equation*}\lim_{s\nearrow1}
\int_{\partial E\cap \Omega^c\cap U_r}\mathcal{A}^s_{E,\Omega\setminus U_r}(x) {\mathcal{X}}(x)\cdot \nu_{\partial E}(x)\,d\mathscr{H}^{n-1}_x=0.
\end{equation*}

As a result, since
$$\mathcal{A}^s_{E,\Omega}=
\mathcal{A}^s_{E,\Omega\setminus U_r}+\mathcal{A}^s_{E,\Omega\cap U_r},$$
we conclude that
\begin{equation}\label{012oejfgrg203hyhaswderg254t}\begin{split}&\lim_{s\nearrow1}
\int_{\partial E\cap \Omega^c\cap U_r}\mathcal{A}^s_{E,\Omega}(x) {\mathcal{X}}(x)\cdot \nu_{\partial E}(x)\,d\mathscr{H}^{n-1}_x\\&\quad
=\lim_{s\nearrow1}
\int_{\partial E\cap \Omega^c\cap U_r}\mathcal{A}^s_{E,\Omega\cap U_r}(x) {\mathcal{X}}(x)\cdot \nu_{\partial E}(x)\,d\mathscr{H}^{n-1}_x.
\end{split}
\end{equation}

Now, we utilize the coarea formula on manifolds, see~\cite[Theorem~3.13]{MR1775760}, and we have that
\begin{equation}\label{0owdfjv:Pi1wkdn}\begin{split}&
\int_{\partial E\cap \Omega^c\cap U_{r}}\mathcal{A}^s_{E,\Omega\cap U_r}(x) {\mathcal{X}}(x)\cdot \nu_{\partial E}(x)\,d\mathscr{H}^{n-1}_x\\&\qquad
=
\int_{\{\omega\cdot X=0\}\cap\{X_1>0\}\cap Q_{r}}\mathcal{A}^s_{E,\Omega\cap U_r}(x) 
\mathscr{V}(x)\,d\mathscr{H}^{n-1}_X,
\end{split}
\end{equation}where the notation~$X=T(x)$ is understood and
we have set, for convenience, 
\begin{equation*}
	\mathscr{V}(x)\coloneqq \frac{{\mathcal{X}}(x)\cdot \nu_{\partial E}(x)}{|\det DT\vert_{\partial E}(x)|},
\end{equation*}
with $T\vert_{\partial E}\colon\partial E\cap U_{r}\to \{\omega\cdot X=0\}\cap Q_{r}$ being the restriction of~$T$ to~$\partial E $.

Furthermore, we apply Lemma~\ref{saxjoml3rf} with~$F\coloneqq E$ and~$F\coloneqq E^c$
(in the latter case, $\omega$ gets replaced by~$-\omega$). Thus, 
up to renaming~$C>0$, we find that
\begin{equation}\label{0qoj9375wdlfv203rt-ty}
\begin{split}&
\big|\mathcal{A}^s_{E,\Omega\cap U_r}(x)-c_{n,s}\big({\mathcal{J}}_1(X)+{\mathcal{J}}_2(X)\big)\big|\\&\qquad=c_{n,s}
\left|
\int_{\Omega\cap U_r}\frac{\chi_{E^c}(y)-\chi_E(y)}{|x-y|^{n+s}}\,dy-{\mathcal{J}}_1(X)-{\mathcal{J}}_2(X)
\right|\\&\qquad\le
C(1-s)r\left(X_1^{1-s}+\frac{(Cr)^{1-s}-X_1^{1-s}}{1-s}\right)
,\end{split}
\end{equation}
where 
\begin{align*}
	\mathcal{J}_1(X)&\coloneqq \int_{\{Y_1<0\}\cap\{\omega\cdot Y>0\}\cap Q_r}|\mathscr{I}(X,Y)|^{-n-s}\frac{|\det DT^{-1}(Y)|}{|X-Y|^{n+s}}\,dY\\
	&\qquad-\int_{\{Y_1<0\}\cap\{\omega\cdot Y<0\}\cap Q_r}|\mathscr{I}(X,Y)|^{-n-s}\frac{|\det DT^{-1}(Y)|}{|X-Y|^{n+s}}\,dY
\end{align*}
and 
\begin{align*}
	\mathcal{J}_2(X)&\coloneqq \int_{\{Y_1<0\}\cap\{\omega\cdot Y>0\}\cap Q_r}\mathscr{B}_s(X,Y)\frac{|\det DT^{-1}(Y)|}{|X-Y|^{n+s}}\,dY\\
	&\qquad-\int_{\{Y_1<0\}\cap\{\omega\cdot Y<0\}\cap Q_r}\mathscr{B}_s(X,Y)\frac{|\det DT^{-1}(Y)|}{|X-Y|^{n+s}}\,dY.
\end{align*}

As a consequence of~\eqref{0owdfjv:Pi1wkdn} and~\eqref{0qoj9375wdlfv203rt-ty}, we find that,
up to renaming~$C>0$ line after line,
\begin{align*}
	\Bigg|&\int_{\partial E\cap \Omega^c\cap U_{r}}\mathcal{A}^s_{E,\Omega\cap U_r}(x) {\mathcal{X}}(x)\cdot \nu_{\partial E}(x)\,d\mathscr{H}^{n-1}_x\\
	&\qquad-c_{n,s}\int_{\{\omega\cdot X=0\}\cap\{X_1>0\}\cap Q_{r}}\big(\mathcal{J}_1(X)+\mathcal{J}_2(X)\big)\mathscr{V}(x)\,d\mathscr{H}^{n-1}_X\Bigg|\\ 
	&\le C(1-s)r\int_{\{\omega\cdot X=0\}\cap\{X_1>0\}\cap Q_{r}}
	\left(X_1^{1-s}+\frac{(Cr)^{1-s}-X_1^{1-s}}{1-s}\right)|\mathscr{V}(x)| \,d\mathscr{H}^{n-1}_X	\\
	&\leq Cr\int_{\{\omega\cdot X=0\}\cap\{X_1>0\}\cap Q_{r}\cap 		\operatorname{spt} {\mathcal{X}}}
\big((1-s)r^{1-s}+(Cr)^{1-s}-X_1^{1-s}\big)\,d\mathscr{H}^{n-1}_X.
\end{align*}
We observe that this quantity
is infinitesimal as~$s\nearrow1$, thanks to the Dominated Convergence Theorem, and as a result we obtain that
\begin{equation}\label{ahasiw}
	\begin{split}
		&\lim_{s\nearrow1}\int_{\partial E\cap \Omega^c\cap U_{r}}\mathcal{A}^s_{E,\Omega\cap U_r}(x) {\mathcal{X}}(x)\cdot \nu_{\partial E}(x)\,d\mathscr{H}^{n-1}_x\\
	&=\lim_{s\nearrow1}c_{n,s}\int_{\{\omega\cdot X=0\}\cap\{X_1>0\}\cap Q_{r}}\big(\mathcal{J}_1(X)+\mathcal{J}_2(X)\big)\mathscr{V}(x)\,d\mathscr{H}^{n-1}_X.
	\end{split}
\end{equation}

We remark that, thanks to the fact that~$DS(0)=0$, we have that,
if~$X\in Q_r$,
\begin{equation*}
\big| \,|\mathscr{I}(X,Y)|^{-n-s}- 1\big|\le C r,
\end{equation*}
for some~$C>0$ uniform with respect to~$X$, $Y\in Q_r$.

Hence, setting~$
\mathcal{J}_\ast(X):=\mathcal{J}^1_\ast(X)-\mathcal{J}^2_\ast(X)$,
with
\begin{eqnarray*}
	\mathcal{J}^1_\ast(X)&\coloneqq& \int_{\{Y_1<0\}\cap\{\omega\cdot Y>0\}\cap Q_r}\frac{|\det DT^{-1}(Y)|}{|X-Y|^{n+s}}\,dY\\
	{\mbox{and }}\quad \mathcal{J}^2_\ast(X)&\coloneqq& 
\int_{\{Y_1<0\}\cap\{\omega\cdot Y<0\}\cap Q_r}\frac{|\det DT^{-1}(Y)|}{|X-Y|^{n+s}}\,dY,
\end{eqnarray*}
we find that
\begin{equation}\label{widhiwudeh}\begin{split}&
|\mathcal{J}_1(X)-\mathcal{J}_\ast(X)|\\
&=\Bigg|
\int_{\{Y_1<0\}\cap\{\omega\cdot Y>0\}\cap Q_r}|\mathscr{I}(X,Y)|^{-n-s}\frac{|\det DT^{-1}(Y)|}{|X-Y|^{n+s}}\,dY-\mathcal{J}^1_\ast(X)\\
	&\qquad-\int_{\{Y_1<0\}\cap\{\omega\cdot Y<0\}\cap Q_r}|\mathscr{I}(X,Y)|^{-n-s}\frac{|\det DT^{-1}(Y)|}{|X-Y|^{n+s}}\,dY
	+\mathcal{J}^2_\ast(X)\Bigg|\\
&\leq
\left|
\int_{\{Y_1<0\}\cap\{\omega\cdot Y>0\}\cap Q_r}\big|\,|\mathscr{I}(X,Y)|^{-n-s}-1\big| \frac{|\det DT^{-1}(Y)|}{|X-Y|^{n+s}}\,dY\right|\\&\qquad+\left|
\int_{\{Y_1<0\}\cap\{\omega\cdot Y<0\}\cap Q_r}\big|\,
|\mathscr{I}(X,Y)|^{-n-s}-1\big|
\frac{|\det DT^{-1}(Y)|}{|X-Y|^{n+s}}\,dY\right|\\
&\leq Cr
\int_{\{Y_1<0\}\cap\{\omega\cdot Y>0\}\cap Q_r}\frac{|\det DT^{-1}(Y)|}{|X-Y|^{n+s}}\,dY\\&\qquad+Cr
\int_{\{Y_1<0\}\cap\{\omega\cdot Y<0\}\cap Q_r}
\frac{|\det DT^{-1}(Y)|}{|X-Y|^{n+s}}\,dY.
\end{split}
\end{equation} 

Furthermore, by~\eqref{012oejfgrg203hyhaswderg254t} and~\eqref{ahasiw} we get that 
\begin{equation}\label{vbcnxieowyr4uiw}
	\lim_{s\nearrow1}
\int_{\partial E\cap \Omega^c\cap U_r}\mathcal{A}^s_{E,\Omega}(x) {\mathcal{X}}(x)\cdot \nu_{\partial E}(x)\,d\mathscr{H}^{n-1}_x=\mathrm{I}+\mathrm{II}+\mathrm{III}
\end{equation}
where 
\begin{align*}
	\mathrm{I}&:=\lim_{s\nearrow1}c_{n,s}\int_{\{\omega\cdot X=0\}\cap\{X_1>0\}\cap Q_{r}}\mathcal{J}_\ast(X)\mathscr{V}(x)\,d\mathscr{H}^{n-1}_X,\\
	\mathrm{II}&:=\lim_{s\nearrow1}c_{n,s}\int_{\{\omega\cdot X=0\}\cap\{X_1>0\}\cap Q_{r}}\big(\mathcal{J}_1(X)-\mathcal{J}_\ast(X)\big)\mathscr{V}(x)\,d\mathscr{H}^{n-1}_X,\\
{\mbox{and}}\qquad	\mathrm{III}&:=\lim_{s\nearrow1}c_{n,s}\int_{\{\omega\cdot X=0\}\cap\{X_1>0\}\cap Q_{r}}\mathcal{J}_2(X)\mathscr{V}(x)\,d\mathscr{H}^{n-1}_X.
\end{align*}

We now write~$\mathrm{I}=\mathrm{I}_1-\mathrm{I}_2$, where,
for~$k\in \{1,2\}$,
\begin{equation*}
	\mathrm{I}_k :=
\lim_{s\nearrow1}c_{n,s}\int_{\{\omega\cdot X=0\}\cap\{X_1>0\}\cap Q_{r}}
{\mathcal{J}}^k_\ast(X)\mathscr{V}(x)\,d\mathscr{H}^{n-1}_X.
\end{equation*}
Note that, for every~$\rho_0>0$,
\begin{equation*}
\lim_{s\nearrow1}c_{n,s}\int_{\{\omega\cdot X=0\}\cap\{X_1>0\}\cap Q_{r}}
\left(
\int_{\R^n\setminus B_{\rho_0}(X)}
\frac{dY}{|X-Y|^{n+s}}
\right)|\mathscr{V}(x)|\, d\mathscr{H}^{n-1}_X\le \lim_{s\nearrow1}\frac{C(1-s)}{s\rho_0^s}=0
\end{equation*}
and therefore
\begin{eqnarray*}&&
	\lim_{s\nearrow1}c_{n,s}\int_{\{\omega\cdot X=0\}\cap\{X_1>0\}\cap Q_{r}}\mathcal{J}^1_\ast(X)\mathscr{V}(x)\,
	d\mathscr{H}^{n-1}_X,\\
	&&\qquad=\lim_{s\nearrow 1}c_{n,s}\int_{{\{\omega\cdot X=0\}}\atop{\{X_1>0\}\cap Q_{r/2}}}
\left(
\int_{{\{Y_1<0\}}\atop{\{\omega\cdot Y>0\}\cap Q_r}}
\frac{|\det DT^{-1}(Y)|}{|X-Y|^{n+s}}\,dY
\right)\mathscr{V}(x)\,d\mathscr{H}^{n-1}_X\\
&&\qquad=\lim_{s\nearrow 1}c_{n,s}\int_{{\{\omega\cdot X=0\}}\atop{\{X_1>0\}\cap Q_{r}}}
\left(
\int_{{\{Y_1<0\}}\atop{\{\omega\cdot Y>0\}}}
\frac{|\det DT^{-1}(Y)|}{|X-Y|^{n+s}}\,dY
\right)\mathscr{V}(x)\,d\mathscr{H}^{n-1}_X.
\end{eqnarray*}

Hence, if we use the substitution~$W\coloneqq (Y_1,Y_2-X_2,\dots,Y_n-X_n)=Y-X+X_1e_1$, we see that
\begin{equation*}\begin{split}&\mathrm{I}_1=
\lim_{s\nearrow1}c_{n,s}\int_{\{\omega\cdot X=0\}\cap\{X_1>0\}\cap Q_{r}}
{\mathcal{J}}^1_\ast(X)\mathscr{V}(x)\,d\mathscr{H}^{n-1}_X\\&=
\lim_{s\nearrow1}c_{n,s}\int_{\{\omega\cdot X=0\}\cap\{X_1>0\}\cap Q_{r}}
\left(
\int_{{\{W_1<0\}}\atop{\{\omega\cdot W+X_1\sin\vartheta>0\}}}
\frac{|\det DT^{-1}(W+X-X_1e_1)|}{|W-X_1e_1|^{n+s}}\,dW
\right)\mathscr{V}(x)\,d{\mathscr{H}}^{n-1}_X
.\end{split}\end{equation*}
Accordingly, substituting for~$Z\coloneqq \frac{ W}{X_1}$ and letting~$\omega_0\coloneqq (-\sin\vartheta,\cos\vartheta)$,
\begin{equation*}
\begin{split}
\mathrm{I}_1&=
\lim_{s\nearrow1}c_{n,s}\int_{\{\omega\cdot X=0\}\cap\{X_1>0\}\cap Q_{r}}
\left(
\int_{{\{Z_1<0\}}\atop{\{\omega\cdot Z+\sin\vartheta>0\}}}
\frac{|\det DT^{-1}(X_1Z+X-X_1e_1)|}{|Z-e_1|^{n+s}}\,dZ
\right)\,\frac{\mathscr{V}(x)\,d{\mathscr{H}}^{n-1}_X}{X_1^s}\\
&=
\lim_{s\nearrow1}c_{n,s}\int_{\{\omega_0\cdot \zeta=0\}\cap\{\zeta_1>0\}\cap Q_{r}}
\left(
\int_{{\{Z_1<0\} }\atop{\{\omega\cdot Z+\sin\vartheta>0\}}}
\frac{|\det DT^{-1}(\zeta_1Z+X-\zeta_1e_1)|}{|Z-e_1|^{n+s}}\,dZ
\right)\,\frac{\mathscr{V}(x)\,d{\mathscr{H}}^{n-1}_X}{\zeta_1^s}\\
&=
\lim_{s\nearrow1}c_{n,s}\int_{\{\omega_0\cdot \zeta=0\}\cap\{\zeta_1>0\}\cap Q_{r}}
{\mathcal{G}}_s(\zeta,X'')\mathscr{V}(x)
\,\frac{d{\mathscr{H}}^{1}_\zeta \,d{\mathscr{H}}^{n-2}_{X''}}{\zeta_1^s}\\
&=
\lim_{s\nearrow1}c_{n,s}\int_{\{\zeta_1>0\}\cap\{\zeta_2=\zeta_1\tan\vartheta\}\cap Q_{r}}
{\mathcal{G}}_s(\zeta,X'')\mathscr{V}(x)
\,\frac{d\zeta_1\,d{\mathscr{H}}^{n-2}_{X''}}{\zeta_1^s},
\end{split}\end{equation*}
where we use the intermediate notation~$X=(\zeta_1,X'',\zeta_2)\in\R\times\R^{n-2}\times\R$
and~$\zeta=(\zeta_1,\zeta_2)\in\R^2$ and set
\begin{equation*}
	{\mathcal{G}}_s(\zeta,X'')\coloneqq \int_{{\{Z_1<0\}}\atop{\{\omega\cdot Z+\sin\vartheta>0\}}}
\frac{|\det DT^{-1}(\zeta_1Z+X-\zeta_1e_1)|}{|Z-e_1|^{n+s}}\,dZ.
\end{equation*}

It is now convenient to change variable~$\tau\coloneqq \zeta_1^{1-s}$ to find that
\begin{equation*}\begin{split}&
\mathrm{I}_1=
\lim_{s\nearrow1}\int_{{\{\tau\in(0,r^{1-s})\}\cap\{\zeta_2=\tau^{1/(1-s)}\tan\vartheta\}}\atop{\cap \{|\zeta_2|<r \}\cap\{|X''|_\infty <r\}}}
{\mathcal{G}}_s(\tau^{1/(1-s)},\zeta_2,X'')\mathscr{V}(x)
\,d\tau\,d{\mathscr{H}}^{n-2}_{X''}.
\end{split}\end{equation*}
Since~${\mathcal{G}}_s$ is bounded uniformly in~$s$ (and compactly supported, since so is~${\mathcal{X}}$), we can now use the Dominated Convergence Theorem and conclude that
\begin{equation}\label{ADSAFD:qewrefeg.r:01}\begin{split}
&\mathrm{I}_1=
\int_{\{\tau\in(0,1)\}\cap\{|X''|_\infty <r\}}
{\mathcal{G}}_1(0,0,X'')\mathscr{V}(T^{-1}(0,X'',0))
\,d\tau\,d{\mathscr{H}}^{n-2}_{X''}
\\&=\int_{\{|X''|_\infty <r\}}
{\mathcal{G}}_1(0,0,X'')\mathscr{V}(T^{-1}(0,X'',0))\,d{\mathscr{H}}^{n-2}_{X''}\\&=\int_{\{|X''|_\infty <r\}}
\left(\int_{ {\{Z_1<0\}}\atop{\{\omega\cdot Z+\sin\vartheta>0\}}}
\frac{|\det DT^{-1}(0,X'',0)|}{|Z-e_1|^{n+s}}\,dZ\right)
\frac{{\mathcal{X}}(T^{-1}(0,X'',0))\cdot \nu_{\partial E}(T^{-1}(0,X'',0))}{|\det DT(T^{-1}(0,X'',0))|}\,d{\mathscr{H}}^{n-2}_{X''}
\\&=\int_{\{|X''|_\infty <r\}}
\left(\int_{{\{Z_1<0\}}\atop{\{\omega\cdot Z+\sin\vartheta>0\}}}
\frac{dZ}{|Z-e_1|^{n+s}}\right)
{\mathcal{X}}(T^{-1}(0,X'',0))\cdot \nu_{\partial E}(T^{-1}(0,X'',0))\,d{\mathscr{H}}^{n-2}_{X''}.
\end{split}\end{equation}

Along the same vein, 
\begin{equation}\label{ADSAFD:qewrefeg.r:01BIS}
	\mathrm{I}_2=\int_{\{|X''|_\infty <r\}}
\left(\int_{{\{Z_1<0\}}\atop{\{\omega\cdot Z+\sin\vartheta<0\}}}
\frac{dZ}{|Z-e_1|^{n+s}}\right)
{\mathcal{X}}(T^{-1}(0,X'',0))\cdot \nu_{\partial E}(T^{-1}(0,X'',0))\,d{\mathscr{H}}^{n-2}_{X''}.
\end{equation}

Furthermore, if~$\widetilde\omega:=(\sin\vartheta,0,\dots,0,\cos\vartheta)$ and~$\widetilde Z:=(Z_1,\dots,Z_{n-1},-Z_n)$,
\begin{eqnarray*}&&
\int_{{\{Z_1<0\}}\atop{\{\widetilde\omega\cdot Z-\sin\vartheta>0\}}}
\frac{dZ}{|Z-e_1|^{n+s}}=\int_{{\{Z_1<0\}}\atop{\{Z_1\sin\vartheta+Z_n\cos\vartheta-\sin\vartheta>0\}}}
\frac{dZ}{|Z-e_1|^{n+s}}\\&&\quad=
\int_{{\{\widetilde Z_1<0\}}\atop{\{\widetilde Z_1\sin\vartheta- \widetilde Z_n\cos\vartheta-\sin\vartheta>0\}}}
\frac{d\widetilde Z}{|\widetilde Z-e_1|^{n+s}}=\int_{{\{\widetilde Z_1<0\}}\atop{\{\omega\cdot \widetilde Z+\sin\vartheta<0\}}}
\frac{d\widetilde Z}{|\widetilde Z-e_1|^{n+s}},
\end{eqnarray*}
leading to
\begin{eqnarray*}&&
\int_{{\{Z_1<0\}}\atop{\{\omega\cdot Z+\sin\vartheta>0\}}}
\frac{dZ}{|Z-e_1|^{n+s}}-\int_{{\{Z_1<0\}}\atop{\{\omega\cdot Z+\sin\vartheta<0\}}}
\frac{dZ}{|Z-e_1|^{n+s}}\\&&\quad=\int_{{\{Z_1<0\}}\atop{{\{\omega\cdot Z+\sin\vartheta>0\}}\atop{\{\widetilde\omega\cdot Z-\sin\vartheta<0\}}}}
\frac{dZ}{|Z-e_1|^{n+s}}=\int_{{\{Z_1<0\}}\atop{
\{Z_n\in(
(Z_1-1)\tan\vartheta,(1-Z_1)\tan\vartheta
)\}
}}\frac{dZ}{|Z-e_1|^{n+s}}.
\end{eqnarray*}
Using this, together with~\eqref{ADSAFD:qewrefeg.r:01} and~\eqref{ADSAFD:qewrefeg.r:01BIS}, we find that
\begin{eqnarray*}
&&\mathrm{I}=\int_{\{|X''|_\infty <r\}}
\left(\int_{{\{Z_1<0\}}\atop{
\{Z_n\in(
(Z_1-1)\tan\vartheta,(1-Z_1)\tan\vartheta
)\}
}}\frac{dZ}{|Z-e_1|^{n+s}}\right)\\&&\qquad\qquad\qquad
{\mathcal{X}}(T^{-1}(0,X'',0))\cdot \nu_{\partial E}(T^{-1}(0,X'',0))\,d{\mathscr{H}}^{n-2}_{X''}.
\end{eqnarray*}

We observe that, setting~$\psi:=\frac\pi2-\vartheta$,
$$ \int_{{\{Z_1<0\}}\atop{\{Z_n\in(
(Z_1-1)\tan\vartheta,(1-Z_1)\tan\vartheta)\}
}}\frac{dZ}{|Z-e_1|^{n+s}}=g(\psi),$$
and so
$$ \mathrm{I}=\int_{\{|X''|_\infty <r\}}
g(\psi)
{\mathcal{X}}(T^{-1}(0,X'',0))\cdot \nu_{\partial E}(T^{-1}(0,X'',0))\,d{\mathscr{H}}^{n-2}_{X''}.$$

Now, using the coarea formula on manifolds, see~\cite[Theorem~3.13]{MR1775760}, we get that
\begin{eqnarray*}
\mathrm{I}=
\int_{\partial E\cap\partial\Omega\cap U_r}g(\psi)
{\mathcal{X}}(x'')\cdot \nu_{\partial E}(x'')\big|\det DT\vert_{\partial E\cap\partial\Omega}(x'')\big|d\mathscr{H}^{n-2}_{x''},
\end{eqnarray*}
with the notation~$x''=T^{-1}(0,X'',0)$.

Note that by smoothness of $g$ we have 
\begin{equation*}
	|g(\psi)-g(\psi_{x''})|\leq C|x''|\leq Cr
\end{equation*}
and similarly, since $T$ is the identity in the origin, 
\begin{equation*}
	\big|1-|\det DT\vert_{\partial E\cap\partial\Omega}(x'')|\big|\leq C|x''|\leq Cr.
\end{equation*}
Also, by the regularity of $\partial E\cap\partial\Omega$ we have 
\begin{equation*}
	\mathscr{H}^{n-2}(\partial E\cap\partial\Omega\cap U_r)=O( r^{n-2}).
\end{equation*}
All in all, we obtain that 
\begin{equation}\label{deicniecn}
	\mathrm{I}=\int_{\partial E\cap\partial\Omega\cap U_r}g(\psi_{x''}){\mathcal{X}}(x'')\cdot \nu_{\partial E}(x'')d\mathscr{H}^{n-2}_{x''}+O(r^{n-1}).
\end{equation}

We now take care of the term~$\mathrm{II}$. For this,
we observe that 
\begin{equation}\label{aggrt78thew2wedfc56tyhj}
{\mbox{if~$X_1>0>Y_1$ then~$|X-Y|\ge X_1-Y_1\ge X_1$,}}
\end{equation}
and therefore
$$ \int_{\{Y_1<0\}\cap Q_r}\frac{dY}{|X-Y|^{n+s}}\le
\int_{|X-Y|\ge X_1}\frac{dY}{|X-Y|^{n+s}}\le\frac{C}{sX_1^s}.
$$
{F}rom this and~\eqref{widhiwudeh} we deduce that
\begin{equation*}
|\mathcal{J}_1(X)-\mathcal{J}_\ast(X)|
\leq \frac{Cr}{sX_1^s}.
\end{equation*}
As a result,
\begin{eqnarray*}&&
\left|\int_{\{\omega\cdot X=0\}\cap\{X_1>0\}\cap Q_{r}}\big(\mathcal{J}_1(X)-\mathcal{J}_\ast(X)\big)\mathscr{V}(x)\,d\mathscr{H}^{n-1}_X\right|\\&&\quad\le \frac{Cr}{s}
\int_{\{\omega\cdot X=0\}\cap\{X_1>0\}\cap Q_{r}}\frac{d\mathscr{H}^{n-1}_X}{X_1^s}
\le \frac{Cr^{n-1}}{s}\int_{0}^r\frac{dX_1}{X_1^s}
= \frac{Cr^{n-s}}{s(1-s)}
\end{eqnarray*}  and thus
\begin{equation}\label{deicniecn2}
\mathrm{II}=O(r^{n-1}).
\end{equation}

To estimate $\mathrm{III}$, note that~$|\mathscr{B}_s(X,Y)|\leq C|X-Y|$ and therefore
\begin{align*}
|\mathcal{J}_2(X)|\leq C\int_{\{Y_1<0\}\cap Q_r}\frac{dY}{|X-Y|^{n+s-1}}.
\end{align*}
Recalling also~\eqref{aggrt78thew2wedfc56tyhj} we thereby obtain that
\begin{equation*}
|\mathcal{J}_2(X)|\le C \int_{|X-Y|\ge X_1}
\frac{dY}{|X-Y|^{n+s-1}}\le \frac{C X_1^{1-s}}{1-s}.
\end{equation*}
It follows that 
\begin{align*}&
	\left|c_{n,s}\int_{\{\omega\cdot X=0\}\cap\{X_1>0\}\cap Q_{r}}\mathcal{J}_2(X)\mathscr{V}(x)\,d\mathscr{H}^{n-1}_X\right|\\
	&\qquad\leq
C\int_{\{\omega\cdot X=0\}\cap\{X_1>0\}\cap Q_{r}}
X_1^{1-s}\,d\mathscr{H}^{n-1}_X\\
&\qquad\leq Cr^{1-s}
\mathscr{H}^{n-1}\Big(
\{\omega\cdot X=0\}\cap\{X_1>0\}\cap Q_{r}\Big)\\
&\qquad\leq Cr^{1-s+n-1}\\
&\qquad\leq  Cr^{n-1}
\end{align*}
from which we deduce that 
\begin{equation}\label{deicniecn3}
\mathrm{III}=O(r^{n-1}).
\end{equation}

Gathering together~\eqref{deicniecn}, \eqref{deicniecn2} and~\eqref{deicniecn3}, and recalling~\eqref{vbcnxieowyr4uiw},
we obtain the desired result.
\end{proof}

With~Lemma \ref{lem:ehfwbuebd} we can now
complete the proof of Theorem~\ref{prop: s to 1}.

\begin{proof}[Proof of Theorem \ref{prop: s to 1}]
The proof follows from a covering argument, whose details are as follows.

Given~$r>0$ sufficiently small, we denote by~$T_{r}(\partial E\cap \partial \Omega)$ the $r$-tubular neighbourhood of $\partial E\cap \partial \Omega$ obtained by local diffeomorphisms with~$Q_r$
as described at the beginning of Section~\ref{sec: lim s1}.
More precisely,
we consider a collection of disjoint open sets~$U_{r}^j$, with~$j=1,\dots,N_r$, such that 
	\begin{equation*}
 T_r(\partial E\cap \partial \Omega)=\bigcup_{j=1}^{N_r}U^j_r=:U_r,
	\end{equation*}
up to sets of null measure.
We point out that~$N_r= O(r^{2-n})$. 
	
By Lemma~\ref{lem:ehfwbuebd}, we have that
	\begin{equation}\label{bvcnxmo387448trghusjafgk}\begin{split}
	&	\lim_{s\nearrow 1}\int_{\partial E\cap \Omega^c\cap U_r}\mathcal{A}^s_{E,\Omega}(x) {\mathcal{X}}(x)\cdot \nu_{\partial E}(x)\,d\mathscr{H}^{n-1}_x\\
		&\qquad=\lim_{s\nearrow 1}\sum_{j=1}^{N_r}\int_{\partial E\cap \Omega^c\cap U_r}\mathcal{A}^s_{E,\Omega}(x) {\mathcal{X}}(x)\cdot \nu_{\partial E}(x)\,d\mathscr{H}^{n-1}_x\\
		&\qquad=\sum_{j=1}^{N_r}\int_{\partial E\cap \partial \Omega\cap U_r}g(\psi_{x'})\mathcal{X}(x')\cdot\nu_{\partial E}(x')\, d{\mathscr{H}}^{n-2}_{x'}+O(N_rr^{n-1})\\
		&\qquad=\int_{\partial E\cap \partial \Omega\cap U_r}g(\psi_{x'})\mathcal{X}(x')\cdot\nu_{\partial E}(x')\,d{\mathscr{H}}^{n-2}_{x'}+O(N_rr^{n-1})\\
&\qquad=\int_{\partial E\cap \partial \Omega}g(\psi_{x'})\mathcal{X}(x')\cdot\nu_{\partial E}(x')\,d{\mathscr{H}}^{n-2}_{x'}+O(r).
	\end{split}\end{equation}
	
We also observe that 
	\begin{equation*}
		\lim_{s\nearrow1}
\left|\int_{(\partial E\cap \Omega^c)\setminus U_r}\mathcal{A}^s_{E,\Omega}(x) {\mathcal{X}}(x)\cdot \nu_{\partial E}(x)\,d\mathscr{H}^{n-1}_x\right|\leq \lim_{s\nearrow1}\frac{C\,c_{n,s}}{r^{n+s}}=0.
	\end{equation*}
	
{F}rom this and~\eqref{bvcnxmo387448trghusjafgk}, we infer that
\begin{eqnarray*}
\lim_{s\nearrow 1}\int_{\partial E\cap \Omega^c}\mathcal{A}^s_{E,\Omega}(x) {\mathcal{X}}(x)\cdot \nu_{\partial E}(x)\,d\mathscr{H}^{n-1}_x=\int_{\partial E\cap \partial \Omega}g(\psi_{x'})\mathcal{X}(x')\cdot\nu_{\partial E}(x')\,d{\mathscr{H}}^{n-2}_{x'}+O(r).
\end{eqnarray*}
The desired result now follows by sending~$r\searrow0$.
\end{proof}

\section{Free boundaries without free boundaries and proof of Theorem~\ref{thm: rad symm sFBMS}}\label{sec: FBWFB}
Here we construct an example of free boundary nonlocal minimal surface~$E$ in the unit ball such that~$\partial E\cap \partial\Omega=\varnothing$, proving Theorem~\ref{thm: rad symm sFBMS}.
We begin our construction with the following preliminary result.

\begin{lemma}\label{lem: R*}
	For any~$s\in (0,1)$ there exists~$R_*=R_*(s)>1$ such that
\begin{equation}\label{ytrewqasdsfgclaim1}
\H^s_{B_{R_*}\setminus B_1}(x)=0 \quad{\mbox{ for }} x\in\partial B_1.\end{equation}
	
	Moreover, 
	\begin{equation}\label{ytrewqasdsfgclaim2}
		\lim_{s\searrow 0}{R_*}(s)=+\infty.
	\end{equation}
	
	\begin{proof}
We point out that, by symmetry, the claim in~\eqref{ytrewqasdsfgclaim1}
is established if we show that
\begin{equation*}
\H^s_{B_{R_*}\setminus B_1}(e_1)=0.
\end{equation*}
To check this, we define 
		\begin{equation*}
			f_s(R)\coloneqq \H^s_{B_R\setminus B_1}(e_1)=c_{n,s}\pv\int_{\R^n}\frac{\chi_{(B_R\setminus B_1)^c}(y)-\chi_{B_R\setminus B_1}(y)}{|e_1-y|^{n+s}}dy.
		\end{equation*} 
		
		First, note that~$f$ is continuous in~$(1,+\infty)$. Indeed, for any~$1<R_1\leq R_2$,
		\begin{equation*}
			|f_s(R_1)-f_s(R_2)|=2c_{n,s}\int_{\R^n}\frac{\chi_{B_{R_2}\setminus B_{R_1}}(y)}{|e_1-y|^{n+s}}dy
			\leq \frac{2c_{n,s}\omega_n}{(R_1-1)^{n+s}}(R_2^n-R_1^n)
		\end{equation*}
		from which continuity follows. 
		
		Next, we observe that 
		\begin{equation}\label{mnbvcxz2345678876543}
			\lim_{R\searrow 1}f_s(R)\in(0,+\infty].
		\end{equation}
Indeed,
\begin{eqnarray*}&&\lim_{R\searrow 1}
f_s(R)=\lim_{R\searrow 1}c_{n,s}\pv\int_{\R^n}\frac{\chi_{B_R^c}(y)+\chi_{B_1}(y)-\chi_{B_R\setminus B_1}(y)}{|e_1-y|^{n+s}}dy\\
&&\qquad=c_{n,s}\pv\int_{\R^n}\frac{\chi_{B_1^c}(y)+\chi_{B_1}(y)}{|e_1-y|^{n+s}}dy=c_{n,s}\pv\int_{\R^n}\frac{dy}{|e_1-y|^{n+s}},
\end{eqnarray*}		
which proves~\eqref{mnbvcxz2345678876543}.

Moreover,
\begin{equation}\label{mnbvcxz23456788765432}	
\lim_{R\nearrow+ \infty}f_s(R)<0.\end{equation}
Indeed,
\begin{eqnarray*}
&&\lim_{R\nearrow+ \infty}f_s(R)=
\lim_{R\nearrow+ \infty}c_{n,s}\pv\int_{\R^n}\frac{\chi_{(B_R\setminus B_1)^c}(y)-\chi_{B_R\setminus B_1}(y)}{|e_1-y|^{n+s}}dy\\&&\qquad
=c_{n,s}\pv\int_{\R^n}\frac{\chi_{B_1}(y)-\chi_{\R^n\setminus B_1}(y)}{|e_1-y|^{n+s}}dy=-\H^s_{B_1}(e_1)<0,
\end{eqnarray*}
which is~\eqref{mnbvcxz23456788765432}.

As a consequence of~\eqref{mnbvcxz2345678876543}
and~\eqref{mnbvcxz23456788765432},
by continuity, there must exist~${R_*}={R_*}(s)>1$ such that~$f_s({R_*})=0$, namely~$\H^s_{B_{R_*}\setminus B_1}(e_1)=0$, as desired.

		To prove~\eqref{ytrewqasdsfgclaim2}, we first
		claim that 
		\begin{equation}\label{gbniow4u326rdsafcsqeqazedctg09876543}
		{\mbox{for any~$R>1$ there exists~$s_0\in(0,1)$ such that if~$s\in(0,s_0)$ then~$f_s(R)>0$.}}\end{equation}
For this, we write
		$$	f_s(R)=\mathrm{I}_s+\mathrm{II}_s,$$
		where
		\begin{eqnarray*}
	\mathrm{I}_s	&\coloneqq	&
			c_{n,s}\pv\int_{B_{2R}(e_1)}\frac{\chi_{(B_R\setminus B_1)^c}(y)-\chi_{B_R\setminus B_1}(y)}{|e_1-y|^{n+s}}dy\\
			{\mbox{and }}\quad
\mathrm{II}_s&\coloneqq	&c_{n,s}\int_{B_{2R}^c(e_1)}\frac{dy}{|e_1-y|^{n+s}}.
		\end{eqnarray*}
		
We take~$\lambda\in(0,R-1)$ and we notice that
\begin{equation}\label{mnbvcxz123456789poiuytre}\begin{split}
\left|\int_{B_{2R}(e_1)\setminus B_{\lambda}(e_1)}\frac{\chi_{(B_R\setminus B_1)^c}(y)-\chi_{B_R\setminus B_1}(y)}{|e_1-y|^{n+s}}dy\right|
&\le 2\int_{B_{2R}(e_1)\setminus B_{\lambda}(e_1)}\frac{dy}{|e_1-y|^{n+s}}
\\&=\frac{2\omega_n}{s}\left(\frac1{\lambda^s}-\frac1{(2R)^s}\right).
\end{split}\end{equation}

Moreover, we observe that if~$y\in B_{\lambda}(e_1)$ then
$$ |y|\le |y-e_1|+1<\lambda+1<R.$$
Therefore
\begin{equation*}\begin{split}
\int_{B_{\lambda}(e_1)}\frac{\chi_{(B_R\setminus B_1)^c}(y)-\chi_{B_R\setminus B_1}(y)}{|e_1-y|^{n+s}}dy
&=\int_{B_{\lambda}(e_1)}\frac{\chi_{B_1}(y)-\chi_{B_R\setminus B_1}(y)}{|e_1-y|^{n+s}}dy\\& =
-\int_{B_{\lambda}(e_1)\cap  P_\lambda}
\frac{dy}{|e_1-y|^{n+s}},
\end{split}\end{equation*}
where
$$ P_\lambda\coloneqq \Big\{x=(x_1,x')\in\R^n\;{\mbox{ s.t. }}\; |x'|<\lambda
\;{\mbox{and}}\; |x_1-1|\le \lambda -\sqrt{\lambda^2-|x'|^2}\Big\}.$$
Hence, by~\cite[Lemma~3.1]{MR3516886} we conclude that
$$\left|\int_{B_{\lambda}(e_1)}\frac{\chi_{(B_R\setminus B_1)^c}(y)-\chi_{B_R\setminus B_1}(y)}{|e_1-y|^{n+s}}dy\right|\le 
\frac{C}{(1-s)\lambda^{s}},
$$ for some~$C>0$ depending only on~$n$.

{F}rom this and~\eqref{mnbvcxz123456789poiuytre}, we gather that
\begin{eqnarray*}
&&\left|\int_{B_{2R}(e_1)}\frac{\chi_{(B_R\setminus B_1)^c}(y)-\chi_{B_R\setminus B_1}(y)}{|e_1-y|^{n+s}}dy\right|
\le \frac{2\omega_n}{s}\left(\frac1{\lambda^s}-\frac1{(2R)^s}\right)+\frac{C}{(1-s)\lambda^{s}}.
\end{eqnarray*}
	As a result,
$$	|\mathrm{I}_s|\le  
\frac{2\omega_n c_{n,s}}{s}\left(\frac1{\lambda^s}-\frac1{(2R)^s}\right)+\frac{Cc_{n,s}}{(1-s)\lambda^{s}}.$$
Thus, exploiting the limits in~\eqref{eq: cns to 1},
\begin{equation}\label{vieruohdlsn5tly5ioijbl}
\lim_{s\searrow0}|\mathrm{I}_s|\le\lim_{s\searrow0}
\frac{2\omega_n c_{n,s}}{s}\left(\frac1{\lambda^s}-\frac1{(2R)^s}\right)+\frac{Cc_{n,s}}{(1-s)\lambda^{s}}=0.\end{equation}
		
Furthermore, changing variable~$z\coloneqq y-e_1$,	\begin{eqnarray*}
			\mathrm{II}_s=c_{n,s}\int_{B_{2R}^c}\frac{dz}{|z|^{n+s}}=\frac{c_{n,s}\,\omega_n}{s(2R)^s}
		\end{eqnarray*}
		and therefore, recalling also the first limit in~\eqref{eq: cns to 1},
	$$\lim_{s\searrow0}	\mathrm{II}_s=\lim_{s\searrow0}
	\frac{c_{n,s}\,\omega_n}{s(2R)^s}= 8.
		$$
{F}rom this and~\eqref{vieruohdlsn5tly5ioijbl}, we obtain the desired claim in~\eqref{gbniow4u326rdsafcsqeqazedctg09876543}.	 

Also, we point out that~$f_s$ is a decreasing function in~$(1,+\infty)$.
Consequently, if~$f_s(R)>0$ then~${R_*}(s)>R$.
{F}rom this observation and~\eqref{gbniow4u326rdsafcsqeqazedctg09876543}, we deduce that
for any~$R>1$ there exists~$s_0\in(0,1)$ such that if~$s\in(0,s_0)$ then~${R_*}(s)>R$. This entails~\eqref{ytrewqasdsfgclaim2}
and concludes the proof of Lemma~\ref{lem: R*}.
	\end{proof}
\end{lemma}

With this preliminary work, we are now ready to complete the proof of Theorem~\ref{thm: rad symm sFBMS}.

\begin{proof}[Proof of Theorem~\ref{thm: rad symm sFBMS}]
	Let~$R_*$ be the radius arising from Lemma~\ref{lem: R*} and note that, by scaling, for any~$r>0$,
\begin{equation}\label{vbcnxmwo98r478trfiuaf3asecWERDGF00}
	\H^s_{B_{R_*r}\setminus B_r}(x)=r^{-s}\H^s_{B_{R_*}\setminus B_1}\left(\frac{x}{r}\right)=0,\quad {\mbox{ for all }} x\in\partial B_r.
\end{equation}

We also claim that there exists~$s_0\in(0,1)$ such that for all~$s\in(0,s_0)$ there exists~$r\in (1/R_*,1)$ such that
\begin{equation}\label{vbcnxmwo98r478trfiuaf3asecWERDGF0}
c_{n,s}\int_{B_1}\frac{\chi_{(B_{R_*r}\setminus B_r)^c}(y)-\chi_{B_{R_*r}\setminus B_r}(y)}{|x-y|^{n+s}}dy=0,
	\quad\quad  {\mbox{for all }} x\in \partial B_{R_*r}.
\end{equation}
Notice that, by symmetry, \eqref{vbcnxmwo98r478trfiuaf3asecWERDGF0}
is established if we show that
\begin{equation}\label{vbcnxmwo98r478trfiuaf3asecWERDGF}
c_{n,s}\int_{B_1}\frac{\chi_{(B_{R_*r}\setminus B_r)^c}(y)-\chi_{B_{R_*r}\setminus B_r}(y)}{|R_*re_1-y|^{n+s}}dy=0.
\end{equation}
 
To check this, we define
\begin{equation*}
	g_s(r)\coloneqq \int_{B_1}\frac{\chi_{(B_{R_*r}\setminus B_r)^c}(y)-\chi_{B_{R_*r}\setminus B_r}(y)}{|R_*re_1-y|^{n+s}}dy=\int_{B_1}\frac{\chi_{B_r}(y)-\chi_{B_{r}^c}(y)}{|R_*re_1-y|^{n+s}}dy
\end{equation*}
and we see that
\begin{equation}\label{bvwjeoir2q235ytwqdgsxcvsr734uifejkd}
\lim_{r\nearrow 1}g_s(r)
=\int_{B_1}\frac{\chi_{B_1}(y)-\chi_{B_{1}^c}(y)}{|R_*re_1-y|^{n+s}}dy=\int_{B_1}\frac{dy}{|R_*re_1-y|^{n+s}}>0.
\end{equation}

Moreover, by~\eqref{ytrewqasdsfgclaim2} in Lemma~\ref{lem: R*}
we have that there exists~$s_0\in(0,1)$ such that if~$s\in(0,s_0)$ then~$R_*>3$. For such values of the parameter~$s$,
we have that
\begin{equation}\label{bvcn238o4723redfiwytrew}
B_{1/R_*}((1-1/R_*)e_1)\subset B_1\setminus B_{1/R_*}.\end{equation}
Indeed, if~$y\in B_{1/R_*}((1-1/R_*)e_1)$ then
$$ |y|\le \left|y-\left(1-\frac1{R_*}\right)e_1 \right|+1-\frac1{R_*}<
\frac1{R^*}+1-\frac1{R_*}=1
$$ and
$$ |y|\ge \left|\left(1-\frac1{R_*}\right)e_1 \right|-\frac1{R_*}=1-\frac{2}{R_*}>\frac{3}{R_*}-\frac{2}{R_*}=\frac{1}{R_*}.
$$
These considerations establish~\eqref{bvcn238o4723redfiwytrew}.

Now, we deduce from~\eqref{bvcn238o4723redfiwytrew} that
\begin{equation*}\begin{split}
\lim_{r\searrow 1/R_*}g_s(r)&
=\int_{B_1}\frac{\chi_{B_{1/R_*}}(y)-\chi_{B_1\setminus B_{1/R_*}}(y)}{|R_*re_1-y|^{n+s}}dy\\&=
\int_{B_1}\frac{\chi_{B_{1/R_*}}(y)-\chi_{B_{1/R_*}((1-1/R_*)e_1)}(y)-\chi_{(B_1\setminus B_{1/R_*})\setminus B_{1/R_*}((1-1/R_*)e_1)}(y)}{|R_*re_1-y|^{n+s}}dy\\&\le
-\int_{B_1}\frac{\chi_{(B_1\setminus B_{1/R_*})\setminus B_{1/R_*}((1-1/R_*)e_1)}(y)}{|R_*re_1-y|^{n+s}}dy
\\&<0.\end{split}\end{equation*}

{F}rom this and~\eqref{bvwjeoir2q235ytwqdgsxcvsr734uifejkd}
we infer the existence of~$r_*\in(1/R_*,1)$ such that~$g_s(r_*)=0$. 
This completes the proof of~\eqref{vbcnxmwo98r478trfiuaf3asecWERDGF}.

Therefore, from~\eqref{vbcnxmwo98r478trfiuaf3asecWERDGF00}
and~\eqref{vbcnxmwo98r478trfiuaf3asecWERDGF0}
we obtain that~$\partial(B_{R_*r_*}\setminus B_{r_*})$ is a free
boundary~$s$-minimal surface in~$B_1$, as desired.
\end{proof} 

	\section{The volume condition and proofs of Theorem~\ref{thm: volume constraint}
	and Corollary~\ref{CPscf:2er3gr}}\label{sec: volume}
	
	Below is the simple, but instructive, proof of Theorem~\ref{thm: volume constraint}.
	
	\begin{proof}[Proof of Theorem~\ref{thm: volume constraint}]
		We point out that~$\partial E\cap\Omega^c$ is unbounded, and we take a sequence~$x_k\in\partial E\cap\Omega^c$ such that~$|x_k|\to+\infty$ as~$k\to +\infty$.
		
		Multiplying the free boundary condition~\eqref{eq: free boundary} by~$|x_k|^{n+s}$, we find that
		\begin{equation*}
			\int_\Omega\frac{|x_k|^{n+s}}{|x_k-y|^{n+s}}\,\big(\chi_{E^c}(y)-\chi_{E}(y)\big)dy=0.
		\end{equation*}
Thus, by the Dominated Convergence Theorem, we obtain that
		\begin{equation*}
			\int_\Omega\big(\chi_{E^c}(y)-\chi_{E}(y)\big)dy=0,
		\end{equation*}
		which concludes the proof.
	\end{proof}
	
	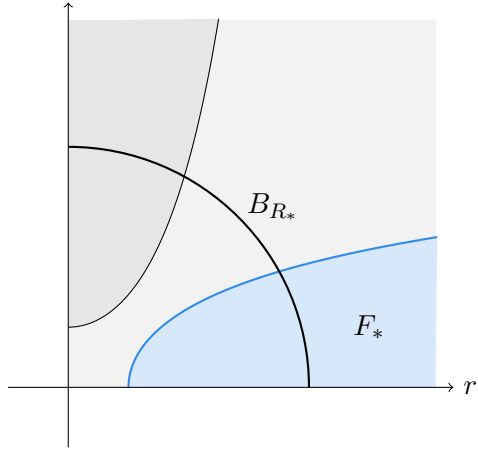
\begin{figure}
\centering
\begin{tikzpicture}[scale=0.8]
  \fill[fill=gray!10] (0,0) -- (6.1,0) -- (6.1,6.1) -- (0,6.1) -- (0,0);
  \fill[fill=gray!20, domain=0:2.5, variable=\x]
    plot ({\x}, {cosh(\x)}) -- (0,6.1) -- cycle;

  \fill[\colorset!20, domain=0:2.5, variable=\y]
    plot ({cosh(\y)}, {\y}) -- (6.1,0) -- cycle;

  \draw[domain=0:2.5, smooth, variable=\x, thin] plot ({\x}, {cosh(\x)});

  \draw[domain=0:2.5, smooth, variable=\y, \colorset, thick]
    plot ({cosh(\y)}, {\y});
    \draw[->] (-1, 0) -- (6.4, 0) node[right] {$r$};
  \draw[->] (0, -1) -- (0, 6.4);
  \fill (5,1) node {$F_*$};
  \fill (3.4,3) node {$B_{R_*}$};
  
  \draw[thick] (4,0) arc[start angle=0, end angle=90, radius=4cm] ;
\end{tikzpicture}
\caption{The catenoid does not split any ball in two parts with equal volume. The volume in the blue part is fully compensated by the volume in the dark grey part, with a positive remainder in light grey.}
\label{fig: catenoid}
\end{figure}

	The work performed so far also allows us to establish
Corollary~\ref{CPscf:2er3gr}.

	\begin{proof}[Proof of Corollary~\ref{CPscf:2er3gr}]
		By the expansion~\eqref{eq: opening C21} it is evident that~$C^{2,1}_s$ does not satisfy the volume condition~\eqref{eq: volume condition} in any ball~$B_R$ when~$s$ is close to~$1$, and therefore
		it cannot be a free boundary~$s$-minimal surface in~$B_R$, thanks to Theorem~\ref{thm: volume constraint}.

We now check that the catenoids~$F_s$ are not
free boundary~$s$-minimal surfaces in any ball~$B_R$ when~$s$ is close to~$1$.

To this aim, we argue by contradiction and suppose that
there exist sequences~$s_k\nearrow 1$ and~$R_k>0$ such that~$F_{s_k}$ is a free boundary~$s_k$-minimal surface in~$B_{R_k}$. 

We point out that
\begin{equation}\label{cbdjnxiuewt7r734rtfugdws0}
{\mbox{$F_{s_k}$ converges locally uniformly to a classical catenoid~$F_*$ as~$k\to+\infty$.}}\end{equation}

Moreover, since~$\partial F_{s_k}$ is unbounded, we are in the position
of applying Theorem~\ref{thm: volume constraint}. In this way, 
we obtain from the volume condition~\eqref{eq: volume condition} that
\begin{equation}\label{bvcxwq57e46udywjxfhsve36wrytdfhsjcyiohlk}
\mathscr{H}^n(F_{s_k}\cap B_{R_k})=\mathscr{H}^n(F_{s_k}^c\cap B_{R_k}).
\end{equation}

We now claim that
\begin{equation}\label{dop3orhfhgd32646etdfgvsiewopgouytrev}
\inf_{k\in{\mathbb{N}}} R_k>0.
\end{equation}
Indeed, suppose by contradiction that, up to a subsequence,
\begin{equation}\label{cbdjnxiuewt7r734rtfugdws2} 
R_k\searrow0\end{equation} and let~$\epsilon>0$ so that
\begin{equation}\label{cbdjnxiuewt7r734rtfugdws}
B_r\cap \bigcup_{x\in\partial F_*} B_{\epsilon}(x)=\varnothing
\end{equation}
for some~$r>0$ (that is, $\epsilon$ is so small that the~$\epsilon$-fattening of the catenoid surface~$\partial F_*$ does not meet the origin).

In light of~\eqref{cbdjnxiuewt7r734rtfugdws0}, for~$k$ sufficiently large, we can also suppose that the distance between~$\partial F_{s_k}\cap B_1$ and~$\partial F_*\cap B_1$
is less than~$\epsilon$. Therefore, by~\eqref{cbdjnxiuewt7r734rtfugdws},
$$ B_r\cap \partial F_{s_k}=\varnothing .$$
Since, by~\eqref{cbdjnxiuewt7r734rtfugdws2}, we know that~$R_k<r$
for large~$k$, we conclude that
$$ B_{R_k}\cap \partial F_{s_k}=\varnothing .$$
As a consequence, one of the sides of~\eqref{bvcxwq57e46udywjxfhsve36wrytdfhsjcyiohlk} is null, and the other strictly positive. This is a contradiction and the proof of~\eqref{dop3orhfhgd32646etdfgvsiewopgouytrev} is thereby complete.

Next, we claim that
\begin{equation}\label{dop3orhfhgd32646etdfgvsiewopgouytrevBIS}
\sup_{k\in{\mathbb{N}}} R_k<+\infty.
\end{equation}
Indeed, suppose by contradiction that~$R_k\to+\infty$. Recall that by the construction in~\cite[Theorem~1]{Davila-delPino-Wei2018}, for~$s$ sufficiently close to~$1$, the fractional catenoid~$F_s$ is the set described as~$\{x=(x',x_3):|x_3|<f(|x'|)\}$, being~$|x'|=\sqrt{x_1^2+x_2^2}$, and 
\begin{equation}\label{eq: expansion fractional catenoid}
		f(r)=\begin{cases}
	\displaystyle	\log(r+\sqrt{r^2-1})+O\left(\frac{r\sqrt{1-s}}{|\log(1-s)|}\right)&\text{if }r<(1-s)^{-1/2},\\
	\displaystyle	r\sqrt{1-s}+O\left(\frac{r\sqrt{1-s}}{|\log(1-s)|}\right)&\text{if }r\geq (1-s)^{-1/2}.
	\end{cases}
\end{equation}

Note that, without loss of generality, we can assume that the catenoid~$F_{s_k}$ is determined exactly by~\eqref{eq: expansion fractional catenoid}
(and not by its rescaling), since otherwise one can take a multiple of~$F_{s_k}$ that is determined by~\eqref{eq: expansion fractional catenoid} and that is a free boundary~$s_k$-minimal surface in a rescaling of~$B_{R_k}$. 

Now, since~$F_{s_k}$ is a free boundary~$s_k$-minimal surface in~$B_{R_k}$, we have that~$\widetilde{F}_{s_k}\coloneqq R_k^{-1}F_{s_k}$ is a free boundary~$s_k$-minimal surface in~$B_1$. Hence, by the expansion~\eqref{eq: expansion fractional catenoid}, we have that~$\widetilde{F}_{s_k}$ is described by~$\{y=(y',y_3)\in B_1:|y_3|<f_k(|y'|)\}$, where 
\begin{equation*}
		f_k(|y'|)=\begin{cases}
	\displaystyle	
	R_{k}^{-1}\log\left(R_k|y'|+\sqrt{R_k^2|y'|^2-1}\right)+O\left(\frac{|y'|\sqrt{1-s_k}}{|\log(1-s_k)|}\right)\\ \qquad\qquad\qquad\qquad\text{if }|y'|<R_k^{-1}(1-s_k)^{-1/2},\\
	\displaystyle 	|y'|\sqrt{1-s_k}+O\left(\frac{|y'|\sqrt{1-s_k}}{|\log(1-s_k)|}\right)
	\\ \qquad\qquad\qquad\qquad
	\text{if }|y'|\geq R_k^{-1}(1-s_k)^{-1/2}.
	\end{cases}
\end{equation*}
Note that~$f_k$ is a sequence of functions converging uniformly to~$0$ in~$B_1'=\{|y'|<1\}$, thus violating~\eqref{bvcxwq57e46udywjxfhsve36wrytdfhsjcyiohlk}.
This contradiction establishes~\eqref{dop3orhfhgd32646etdfgvsiewopgouytrevBIS}.

As a consequence of~\eqref{dop3orhfhgd32646etdfgvsiewopgouytrev} and~\eqref{dop3orhfhgd32646etdfgvsiewopgouytrevBIS}, up to choosing a subsequence, we can assume that~$R_k\to R_*\in(0,+\infty)$.
Then, in light of~\eqref{cbdjnxiuewt7r734rtfugdws0}, we have that~$F_{s_k}$ converges to~$F_*$ in~$B_{R_*}$. This is impossible since~$F_*$ does not satisfy~\eqref{eq: volume condition} in any ball, see Figure~\ref{fig: catenoid}.\end{proof}
	
	\section{Stickiness, boundary regularity and proofs of Theorems~\ref{dj9asofhvkrew8yhmoijyuj}
and~\ref{R-wf0jgvhr90tiuohjtktyu}}\label{sec: stickiness}

This section is devoted to the boundary analysis of free boundary nonlocal minimal surfaces, namely the stickiness statement in
Theorem~\ref{dj9asofhvkrew8yhmoijyuj} and the
boundary behaviour in Theorem~\ref{R-wf0jgvhr90tiuohjtktyu}.

	\begin{proof}[Proof of Theorem~\ref{dj9asofhvkrew8yhmoijyuj}]
	Up to a rigid motion, we can suppose by contradiction that~$E$ sticks to~$\Omega$ from outside at~$0\in\partial E\cap\partial\Omega$. 
	Namely, there exists~$\rho>0$ such that
	\begin{equation}\label{vbncxmowiu4ytr36ewdsqrtfa1qaszx43rfgv}\begin{split}&
{\mbox{either~$\Omega\cap B_\rho\subset  E$ and~$\Omega^c\cap B_\rho\cap E\ne\varnothing$}}\\&
{\mbox{or~$\Omega\cap B_\rho\subset  E^c$ and~$\Omega^c\cap B_\rho\cap E^c\ne\varnothing$.}}\end{split}\end{equation}
In both cases, we have that there exists a sequence of points~$\{x_k\}\subset\partial E\cap \overline\Omega^c$ with~$x_k\searrow0$ and 
		\begin{equation}\label{cbnxklwue8276w1qas4rfv7uj}
			\int_{\Omega}\frac{\chi_{E^c}(y)-\chi_{E}(y)}{|x_k-y|^{n+s}}dy=0.		\end{equation}

Suppose that the first situation in~\eqref{vbncxmowiu4ytr36ewdsqrtfa1qaszx43rfgv} occurs
(the other one being analogous). In this case, we deduce from~\eqref{cbnxklwue8276w1qas4rfv7uj} that
\begin{equation}\label{vbcnxido32y8rtiugdhsj7ri3guwejdb6958746}
			\int_{\Omega\cap B_\rho}\frac{dy}{|x_k-y|^{n+s}}=\int_{\Omega\cap B_\rho}\frac{\chi_{E}(y)-\chi_{E^c}(y)}{|x_k-y|^{n+s}}dy=\int_{\Omega\cap B_\rho^c}\frac{\chi_{E^c}(y)-\chi_{E}(y)}{|x_k-y|^{n+s}}dy.
		\end{equation}
		
Moreover, for~$k$ large enough, we have that~$x_k\in B_{\rho/2}$. Therefore, if~$y\in\Omega\cap B_\rho^c$,
$$ |x_k-y|\ge |y|-|x_k|\ge |y|-\frac{|y|}2=\frac{|y|}2.$$
As a consequence,
$$\left| \int_{\Omega\cap B_\rho^c}\frac{\chi_{E^c}(y)-\chi_{E}(y)}{|x_k-y|^{n+s}}dy\right|\le 2\int_{\Omega\cap B_\rho^c}\frac{dy}{|x_k-y|^{n+s}}\le 2^{n+s+1}\int_{\Omega\cap B_\rho^c}\frac{dy}{|y|^{n+s}}\le C,
 $$ for some~$C>0$ independent of~$k$.

{F}rom this and~\eqref{vbcnxido32y8rtiugdhsj7ri3guwejdb6958746}
we infer that, for~$k$ large enough,
$$ \int_{\Omega\cap B_\rho}\frac{dy}{|x_k-y|^{n+s}}\le C.$$
Then, by Fatou's Lemma,
\begin{equation}\label{vbcnxido32y8rtiugdhsj7ri3guwejdb6958746BIS}
			\int_{\Omega\cap B_\rho}\frac{dy}{|y|^{n+s}}\leq \lim_{k\to +\infty}\int_{\Omega\cap B_\rho}\frac{dy}{|x_k-y|^{n+s}}\leq C.
		\end{equation}
		
On the other hand, we have that
$$ \int_{\Omega\cap B_\rho}\frac{dy}{|y|^{n+s}}=+\infty,$$
in contradiction with~\eqref{vbcnxido32y8rtiugdhsj7ri3guwejdb6958746BIS}.
\end{proof}

Before proving Theorem~\ref{R-wf0jgvhr90tiuohjtktyu}, we need some preliminary statements.

We first show that corners produce an infinite nonlocal mean curvature.
Some care is needed for a statement of this type, because of course symmetric
grids may have vanishing mean curvature.
Also, in our setting the nonlocal mean curvature is not computed exactly
at the corner, but only arbitrarily close to it, and this produces some technical
issues in the integral calculations.

The result that we need goes as follows:

\begin{lemma}\label{sdfertyu:qwdf} Let~$E\subset\R^n$, with~$0\in\partial E$.
Let~$\alpha\in(s,1]$ and~$T:\R^n\to\R^n$ be a diffeomorphism of class~$C^{1,\alpha}$
with~$T(0)=0$, $ DT(0)=\operatorname{Id}$, and
such that~$T(B_r)=B_r$ and~$T(E\cap B_r)=E_1\cup E_2$, where
\begin{eqnarray*}
&& E_1\coloneqq \big\{ {\mbox{$x\in B_r$ s.t. $x_1\ge0$ and $\omega_1\cdot x<0$}}\big\}\\{\mbox{and }}  &&
E_2\coloneqq \big\{ {\mbox{$x\in B_r$ s.t. $x_1\le0$ and $\omega_2\cdot x<0$}}\big\},
\end{eqnarray*}
for some unit vectors~$\omega_1$, $\omega_2\in\R^n$ (see Figure \ref{fig:E1E2}).

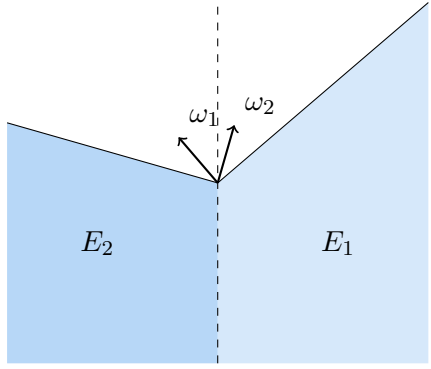
\begin{figure}
	\centering
	\begin{tikzpicture}[scale=.8]
	
		\coordinate (O) at (0,0);
		\coordinate (E1) at (3.5,3);
		\coordinate (BR1) at (3.5,-3);
		\coordinate (E2) at (-3.5,1);
		\coordinate (BR2) at (-3.5,-3);
		\coordinate (BR) at (0,-3);

		\fill[fill=\colorset!20] (O) to (E1) to (BR1) to (BR) to (O);
		
		\fill[fill=\colorset!35] (O) to (E2) to (BR2) to (BR) to (O);
		\draw[black,dashed] (0,-3) to (0,3);
		
		\draw (O) to (E1);
		\draw (O) to (E2);
		
		\fill (2,-1) node {$E_1$};
		\fill (-2,-1) node {$E_2$};
		
		\draw[thick, ->] (O) -- ++(130.6:1cm) node[above right] {$\omega_1$};
		\draw[thick, ->] (O) -- ++(254-180:1cm) node[above right] {$\omega_2$};
		
	\end{tikzpicture}
	\caption{The sets $E_1$ and $E_2$ from Lemma \ref{sdfertyu:qwdf}}
	\label{fig:E1E2}
\end{figure}

Suppose that\footnote{Up to complementary sets, Lemma~\ref{sdfertyu:qwdf}
has a similar statement in which condition~\eqref{sdfertyu:qwdf4}
is replaced by \label{02eorfjgvbg45h54nb}
$$\big\{ {\mbox{$x\in B_r$ s.t. $x_1\ge0$ and $\omega_2\cdot x<0$}}\big\}\subset E_1$$
and the corresponding thesis in~\eqref{sdfertyu:qwdfY} becomes$$ \lim_{T(\partial E_1)\ni x\to0}\int_{\R^n}\frac{\chi_{E^c}(y)-\chi_E(y)}{|x-y|^{n+s}}\,dy
=+\infty.$$}
\begin{equation}\label{sdfertyu:qwdf4} \big\{ {\mbox{$x\in B_r$ s.t. $x_1\le0$ and $\omega_1\cdot x<0$}}\big\}\subset E_2.\end{equation}

Then, either~$\omega_1=\omega_2$ or
\begin{equation}\label{sdfertyu:qwdfY} \lim_{T(\partial E_1)\ni x\to0}\int_{\R^n}\frac{\chi_{E^c}(y)-\chi_E(y)}{|x-y|^{n+s}}\,dy
=-\infty.\end{equation}
\end{lemma}

\begin{proof} 
We point out that, for all~$x$, $y\in B_r$,
\begin{equation*}
\big| T(x)-T(y)-(x-y)\big|=
\left| \int_0^1\Big( DT\big(tx+(1-t)y\big)-{\operatorname{Id}}\Big)(x-y)\,dt
\right|\le
C|x-y|^{1+\alpha},
\end{equation*}for some~$C>0$,
and, in a similar vein, up to freely renaming~$C$,
\begin{equation}\label{q123ert78sdfkojbgl0}
\big| T^{-1}(X)-T^{-1}(Y)-(X-Y)\big|\le
C|X-Y|^{1+\alpha}.
\end{equation} 

We also claim that, for all~$a$, $b\ge0$ and all~$\gamma\ge1$,
\begin{equation}\label{q123ert78sdfkojbgl}
|a^\gamma-b^\gamma|\le\gamma(a+b)^{\gamma-1}|a-b|.
\end{equation}
To check this, without loss of generality, we can assume that~$a>0$ and~$ b>0$, otherwise the result is obvious, and, up to swapping~$a$ and~$b$, that~$a\ge b$.
Then, we let~$c\coloneqq a-b\ge0$ and we find that
\begin{eqnarray*}
&&|a^\gamma-b^\gamma|=(b+c)^\gamma-b^\gamma=
\gamma\int_0^c(b+t)^{\gamma-1}\,dt\le
\gamma (b+c)^{\gamma-1}c\\&&\qquad=
\gamma a^{\gamma-1}(a-b)\le\gamma(a+b)^{\gamma-1}|a-b|,
\end{eqnarray*}
which establishes~\eqref{q123ert78sdfkojbgl}.

By~\eqref{q123ert78sdfkojbgl0} and~\eqref{q123ert78sdfkojbgl}, used here with~$a\coloneqq |X-Y|$, $b\coloneqq |T^{-1}(X)-T^{-1}(Y)|$ and~$\gamma\coloneqq n+s$, we gather that
\begin{equation}\label{8h612wed012wedv203-pgjbko}
\begin{split}&
\Big||X-Y|^{n+s}-|T^{-1}(X)-T^{-1}(Y)|^{n+s}\Big|\\ \le\;&
(n+s) \Big(
|X-Y|+|T^{-1}(X)-T^{-1}(Y)|
\Big)^{n+s-1}\Big||X-Y|-|T^{-1}(X)-T^{-1}(Y)|\Big|\\  \le\;&
C|X-Y|^{n+s-1}\Big||X-Y|-|T^{-1}(X)-T^{-1}(Y)|\Big|\\ \le\;&
C|X-Y|^{n+s-1}\big|(X-Y)-(T^{-1}(X)-T^{-1}(Y))\big|\\ \le\;&
C|X-Y|^{n+s+\alpha}\\ \le\;&
C|T^{-1}(X)-T^{-1}(Y)|^{n+s}|X-Y|^\alpha.
\end{split}
\end{equation}

Now, given~$x\in B_r$, we use the notation~$X\coloneqq T(x)$ and
$$\Phi(X,Y)\coloneqq \frac{|X-Y|^{n+s}}{|T^{-1}(X)-T^{-1}(Y)|^{n+s}}.$$
It follows from~\eqref{8h612wed012wedv203-pgjbko} that
\begin{equation*}
|\Phi(X,Y)-1|=\left| \frac{|X-Y|^{n+s}-|T^{-1}(X)-T^{-1}(Y)|^{n+s}}{|T^{-1}(X)-T^{-1}(Y)|^{n+s}}\right|
\le C |X-Y|^{\alpha}.
\end{equation*}

Hence, if~$E_*\coloneqq E_1\cup E_2$, we find that, for all~$x\in B_r$,
\begin{equation*}
\begin{split}
\Xi(x)&\coloneqq \int_{B_r}\frac{\chi_{E^c}(y)-\chi_E(y)}{|x-y|^{n+s}}\,dy-
\int_{B_r}\Big(\chi_{E_*^c}(Y)-\chi_{E_*}(Y)\Big)\frac{|\det DT^{-1}(Y)|}{|X-Y|^{n+s}}\,dY\\&
=\int_{B_r}\Big( \chi_{E_*^c}(Y)-\chi_{E_*}(Y)\Big)
\frac{|\det DT^{-1}(Y)|\,dY}{|T^{-1}(X)-T^{-1}(Y)|^{n+s}}\\&\qquad\qquad-
\int_{B_r}\Big(\chi_{E_*^c}(Y)-\chi_{E_*}(Y)\Big)\frac{|\det DT^{-1}(Y)|}{|X-Y|^{n+s}}\,dY
\\&=\int_{B_r}\frac{\chi_{E_*^c}(Y)-\chi_{E_*}(Y)}{|X-Y|^{n+s}}\,\Big(\Phi(X,Y)-1\Big)|\det DT^{-1}(Y)|\,dY
\end{split}
\end{equation*}
and consequently
\begin{equation}\label{o2kwfvch5iEhi-yy}
\begin{split}
|\Xi(x)|\le C\int_{B_r}\frac{|\Phi(X,Y)-1|}{|X-Y|^{n+s}}\,dY\le
C\int_{B_r}\frac{|X-Y|^\alpha}{|X-Y|^{n+s}}\,dY\le Cr^{\alpha-s}.
\end{split}
\end{equation}

We let
$$ G\coloneqq \big\{ {\mbox{$x\in\R^n$ s.t. $\omega_1\cdot x<0$}}\big\}$$
and we claim that
\begin{equation}\label{q0owjf:012efvn82ijkwdf0123em9w85g3i37}\left|
\int_{B_r}\Big(\chi_{G^c}(Y)-\chi_{G}(Y)\Big)\frac{|\det DT^{-1}(Y)|}{|X-Y|^{n+s}}\,dY\right|\le Cr^{\alpha-s}.
\end{equation}
To check this, we observe that, by symmetry,
$$ \int_{B_r}\Big(\chi_{G^c}(Y)-\chi_{G}(Y)\Big)\frac{dY}{|X-Y|^{n+s}}=0$$
and accordingly
\begin{eqnarray*}
&&\left|\int_{B_r}\Big(\chi_{G^c}(Y)-\chi_{G}(Y)\Big)\frac{|\det DT^{-1}(Y)|}{|X-Y|^{n+s}}\,dY\right|\\&&\qquad=
\left|\int_{B_r}\Big(\chi_{G^c}(Y)-\chi_{G}(Y)\Big)\frac{|\det DT^{-1}(Y)|-|\det DT^{-1}(X)|}{|X-Y|^{n+s}}\,dY\right|\\&&\qquad\le
\int_{B_r}\frac{\Big||\det DT^{-1}(Y)|-|\det DT^{-1}(X)|\Big|}{|X-Y|^{n+s}}\,dY\\&&\qquad\le C\int_{B_r}\frac{dY}{|X-Y|^{n+s-\alpha}}\\&&\qquad\le Cr^{\alpha-s}
\end{eqnarray*}
and this gives~\eqref{q0owjf:012efvn82ijkwdf0123em9w85g3i37}.

Now suppose that~$\omega_1\ne\omega_2$. Then, the cone
\begin{eqnarray*}&&F\coloneqq \big\{ {\mbox{$x\in B_r$ s.t. $x_1\le0$ and $\omega_2\cdot x<0\le\omega_1\cdot x$}}\big\}\end{eqnarray*}
has positive measure and therefore, up to taking~$r$ smaller if needed,
changing variables~$Z\coloneqq \frac{Y}{|X|}$, and setting~$\widehat X\coloneqq \frac{X}{|X|}$,
\begin{equation*}
\int_{F\cap B_r}\frac{|\det DT^{-1}(Y)|}{|X-Y|^{n+s}}\,dY\ge
c\int_{F\cap B_r}\frac{dY}{|X-Y|^{n+s}}=
\frac{c}{|X|^s}\int_{B_{r/|X|}\cap F}\frac{dZ}{|\widehat X-Z|^{n+s}}\ge\frac{c}{|X|^s},
\end{equation*}
as long as~$|X|$ is small enough (possibly with respect to~$r$), and up to renaming~$c>0$.

{F}rom this, \eqref{o2kwfvch5iEhi-yy}
and~\eqref{q0owjf:012efvn82ijkwdf0123em9w85g3i37}, writing~$E_*=(F\cup G)\cap B_r$,
we deduce that
\begin{equation}\label{9qerf:COqwfdvbngh3W3d56g7E3DF}
\begin{split}&
\int_{B_r}\frac{\chi_{E^c}(y)-\chi_E(y)}{|x-y|^{n+s}}\,dy\\=\;&
\Xi(x)+
\int_{B_r}\Big(\chi_{E_*^c}(Y)-\chi_{E_*}(Y)\Big)\frac{|\det DT^{-1}(Y)|}{|X-Y|^{n+s}}\,dY\\ =\;&\Xi(x)+
\int_{B_r}\Big(\chi_{G_*^c}(Y)-\chi_{G_*}(Y)\Big)\frac{|\det DT^{-1}(Y)|}{|X-Y|^{n+s}}\,dY-2\int_{F\cap B_r}\frac{|\det DT^{-1}(Y)|}{|X-Y|^{n+s}}\,dY\\ 
\le\;& Cr^{\alpha-s}-\frac{c}{|X|^s}.
\end{split}\end{equation}

Also, if~$|x|\le\frac{r}2$,
\begin{eqnarray*}
\left|\int_{B_r^c}\frac{\chi_{E^c}(y)-\chi_E(y)}{|x-y|^{n+s}}\,dy\right|\le
\int_{B_r^c}\frac{dy}{|x-y|^{n+s}}\le C\int_{B_r^c}\frac{dy}{|y|^{n+s}}\le
\frac{C}{r^s}.
\end{eqnarray*}

Combining this and~\eqref{9qerf:COqwfdvbngh3W3d56g7E3DF}, we conclude that
\begin{eqnarray*}
&&\lim_{T(\partial E_1)\ni x\to0}\int_{\R^n}\frac{\chi_{E^c}(y)-\chi_E(y)}{|x-y|^{n+s}}\,dy=
\lim_{E_1\ni X=T(x)\to0}\int_{\R^n}\frac{\chi_{E^c}(y)-\chi_E(y)}{|x-y|^{n+s}}\,dy
\\&&\qquad\le
\lim_{E_1\ni X=T(x)\to0}
\frac{C}{r^s}+
Cr^{\alpha-s}-\frac{c}{|X|^s}=-\infty,
\end{eqnarray*}as desired.
\end{proof}

Below is a useful variation of Lemma~\ref{sdfertyu:qwdf} estimating integral contributions in~$\Omega$:

\begin{figure}
	\begin{center}
		\begin{tikzpicture}[scale=.8]
			
			\coordinate (SW) at (0,-3);
		\coordinate (SE) at (4,-3);
		\coordinate (NW) at (0,3);
		\coordinate (NE) at (4,3);
		\coordinate (O) at (0,0);
		\coordinate (P) at (4,2);
		\coordinate (P2) at (-4,-2);
		\coordinate (S) at (-4,-3);

		\fill[fill=\colorset!10] (P2) to (P) to (SE) to (S) to (P2);
		
		\fill[fill=gray!20, fill opacity=0.3] (SW) to (SE) to (NE) to (NW) to (SW);
		
		\draw[black] (SW) to (NW);

		\draw[black] (P) to (P2);

		\fill (-2,-2) node {$E$};
		\fill (2,2.3) node {$\Omega$};
		
		\draw[thick, ->] (O) -- ++(116.6:1cm) node[above] {$\omega_1$};
		\draw[thick, ->] (O) -- ++(116.6+90:1cm) node[above] {$\varpi$};
			
		\end{tikzpicture}
	\end{center}
	\caption{The setting of Lemma \ref{osjdlcvORmJUfbgb}}
	\label{fig:omega and varpi}
\end{figure}
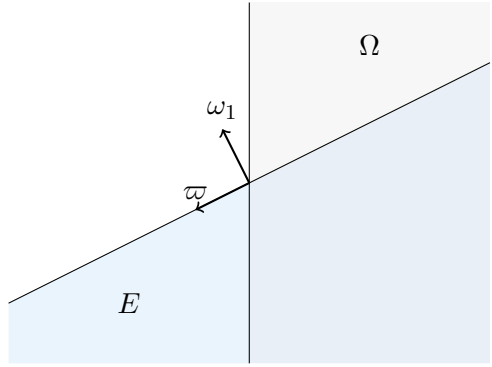

\begin{lemma}\label{osjdlcvORmJUfbgb}
Let $E$ and $\Omega$ be as above and let
\begin{equation}\label{aojsdcknv9012or3jthgbn823b4bVSNIdkf}
\vartheta\in\left(-\frac\pi2,\frac\pi2\right).\end{equation}
Assume that~$0\in\partial E\cap\partial \Omega$ and that there exist~$r>0$ and a diffeomorphism~$T:\R^n\to\R^n$ of class~$C^{1,\alpha}$,
with~$T(0)=0$, $ DT(0)=\operatorname{Id}$,
$T(B_r)=B_r$, $T(B_r\cap\Omega)=B_r\cap\{x_1>0\}$
and
\begin{equation}\label{019qowdjfknVJSmdcKowrefgr02er8}
T(E\cap B_r)=\big\{ {\mbox{$x\in B_r$ s.t. $\omega\cdot x<0$}}\big\}\end{equation}
for some unit vector~$\omega=(-\sin\vartheta,0\dots,0,\cos\vartheta)$.

Let~$\rho_k\in\left(0,\frac{r}2\right)$ be an infinitesimal sequence and~$\varpi\coloneqq (-\cos\vartheta,0,\dots,0,-\sin\vartheta)$ (see Figure \ref{fig:omega and varpi}). Let also~$x_k\coloneqq T^{-1}(\rho_k\varpi)$.

Then, $x_k$ is infinitesimal as~$k\to+\infty$. Also, for large~$k$,
\begin{equation}\label{9CVqGBueowdfhkgv095pyoujm:203er}
x_k\in(\partial E)\cap\overline{\Omega}^c.
\end{equation}

Moreover, if~$\vartheta\in\left(0,\frac\pi2\right)$, then
\begin{equation}\label{m5cfpa6g823lscgokdmc204r-1}
\lim_{k\to+\infty}\int_\Omega\frac{\chi_{E^c}(y)-\chi_E(y)}{|x_k-y|^{n+s}}\,dy=-\infty.\end{equation}

Similarly, if~$\vartheta\in\left(-\frac\pi2,0\right)$, then
\begin{equation}\label{m5cfpa6g823lscgokdmc204r-2}\lim_{k\to+\infty}\int_\Omega\frac{\chi_{E^c}(y)-\chi_E(y)}{|x_k-y|^{n+s}}\,dy=+\infty.\end{equation}\end{lemma}

\begin{proof} We remark that~$\omega\cdot\varpi=0$ and therefore, by~\eqref{019qowdjfknVJSmdcKowrefgr02er8}, for large~$k$ we have that
\begin{equation}\label{9CVqGBueowdfhkgv095pyoujm:203erb}T(x_k)=\rho_k\varpi\in\big\{ {\mbox{$x\in B_r$ s.t. $\omega\cdot x=0$}}\big\}=
T(\partial E\cap B_r).\end{equation}

Furthermore, if~$e_1\coloneqq (1,0,\dots,0)$, we see that~$e_1\cdot\varpi=
-\cos\vartheta<0$, owing to~\eqref{aojsdcknv9012or3jthgbn823b4bVSNIdkf}.
Accordingly, for large~$k$, we have that
$$ T(x_k)=\rho_k\varpi\in\big\{ {\mbox{$x\in B_r$ s.t. $e_1\cdot x<0$}}\big\}=
T(\overline\Omega^c\cap B_r).
$$
{F}rom this and~\eqref{9CVqGBueowdfhkgv095pyoujm:203erb}, we obtain~\eqref{9CVqGBueowdfhkgv095pyoujm:203er}, as desired.

Now we use the notation~$Y\coloneqq T(y)$ and we compute that
\begin{equation*}\begin{split}
\Upsilon_k &\coloneqq \int_{B_r\cap\Omega}
\frac{\chi_{E^c}(y)-\chi_E(y)}{|x_k-y|^{n+s}}\,dy
\\& =\int_{B_r \cap \{Y_1>0\} \cap \{\omega\cdot Y>0\} }
\frac{|\det DT^{-1}(Y)|\,dY}{|T^{-1}(\rho_k\varpi)-T^{-1}(Y)|^{n+s}}\\&\qquad\quad-
\int_{B_r \cap \{Y_1>0\} \cap \{\omega\cdot Y<0\} }
\frac{|\det DT^{-1}(Y)|\,dY }{ |T^{-1}(\rho_k\varpi)-T^{-1}(Y)|^{n+s}}.
\end{split}
\end{equation*}
Substituting for~$Z\coloneqq \frac{Y}{\rho_k}$, we conclude that
\begin{equation}\label{owdjfwrrgh034tugjhb-3rgrhrbnHJiwefhkv}\begin{split}
\rho_k^s\Upsilon_k
&=\int_{ B_{r/\rho_k} \cap \{Z_1>0\} \cap \{\omega\cdot Z>0\} }
\frac{|\det DT^{-1}(\rho_k Z)|\,dZ}{|\rho_k^{1}T^{-1}(\rho_k\varpi)-\rho_k^{-1}T^{-1}(\rho_k Z)|^{n+s}}\\
&\qquad-
\int_{ B_{r/\rho_k} \cap \{Z_1>0\} \cap \{\omega\cdot Z<0\} }
\frac{|\det DT^{-1}(\rho_k Z)|\,dZ}{|\rho_k^{-1}T^{-1}(\rho_k\varpi)-\rho_k^{-1}T^{-1}(\rho_k Z)|^{n+s}}=:\Xi_k.
\end{split}
\end{equation}

We stress that
$$ \rho_k^{-1}T^{-1}(\rho_k Z)=
\rho^{-1}_k \Big( \rho_k DT^{-1}(0) Z+O(\rho_k^{1+\alpha})\Big)=Z+O(\rho_k^\alpha),$$
as well as~$\rho_k^{-1}T^{-1}(\rho_k \varpi)=\varpi+O(\rho_k^\alpha)$.

Hence, recalling~\eqref{aojsdcknv9012or3jthgbn823b4bVSNIdkf},
$$ \lim_{k\to+\infty}
e_1\rho_k^{-1}T^{-1}(\rho_k \varpi)=e_1\cdot\varpi=-\cos\vartheta$$
and thus, for large~$k$, we have that~$e_1\rho_k^{-1}T^{-1}(\rho_k \varpi)\le-\frac{\cos\vartheta}2<0$.

This allows us to use the Dominated Convergence Theorem and deduce that
\begin{equation}\label{lsL94kdpSrgh023r0itgjbMB:1}
\lim_{k\to+\infty}\Xi_k=
\int_{ \{Z_1>0\} \cap \{\omega\cdot Z>0\} }
\frac{dZ}{|\varpi-Z|^{n+s}}-
\int_{ \{Z_1>0\} \cap \{\omega\cdot Z<0\} }
\frac{dZ}{|\varpi-Z|^{n+s}}.\end{equation}

We now consider the reflection~${\mathcal{R}}$ through the hyperplane normal to~$\omega$, namely
$$W\coloneqq {\mathcal{R}}(Z)=Z-2(\omega\cdot Z)\omega.$$
Thus, since~$Z=W-2(\omega\cdot W)\omega$ and~$\omega\cdot\varpi=0$,
\begin{eqnarray*}&&
|\varpi-Z|^2=|(\varpi-W)+2(\omega\cdot W)\omega|^2\\&&\quad=
|\varpi-W|^2+4(\omega\cdot W)^2+4(\omega\cdot W)(\varpi-W)\cdot\omega
=|\varpi-W|^2.\end{eqnarray*}
On this account,
\begin{equation}\label{lsL94kdpSrgh023r0itgjbMB:2}\Lambda\coloneqq 
\int_{ \{Z_1>0\} \cap \{\omega\cdot Z>0\} }
\frac{dZ}{|\varpi-Z|^{n+s}}
=\int_{ \{W_1>2(\omega\cdot W)\omega_1\} \cap \{\omega\cdot W<0\} }
\frac{dW}{|\varpi-W|^{n+s}}.
\end{equation}

Suppose now that~$\vartheta\in\left(0,\frac\pi2\right)$. Then, $\omega_1<0$. Hence, 
$$ \{W_1>2(\omega\cdot W)\omega_1\} \cap \{\omega\cdot W<0\}\subset
\{W_1>0\}.$$
It follows from this observation, \eqref{lsL94kdpSrgh023r0itgjbMB:1} and~\eqref{lsL94kdpSrgh023r0itgjbMB:2} that
\begin{eqnarray*}\lim_{k\to+\infty}\Xi_k&=&\Lambda-
\int_{ \{Z_1>0\} \cap \{\omega\cdot Z<0\} }
\frac{dZ}{|\varpi-Z|^{n+s}}\\&=&-
\int_{ \big\{W_1\in\big(0,2(\omega\cdot W)\omega_1\big]\big\} \cap \{\omega\cdot W<0\} }
\frac{dW}{|\varpi-W|^{n+s}},
\end{eqnarray*}
which is a strictly negative quantity.

This and~\eqref{owdjfwrrgh034tugjhb-3rgrhrbnHJiwefhkv} yield that
\begin{equation}\label{4fGt6hodjf4c619loo3402rfg}
\lim_{k\to+\infty}\Upsilon_k=-\infty.\end{equation}

Also, since~$x_k\in B_{r/2}$,
$$ \left|\int_{B_r^c\cap\Omega}\frac{\chi_{E^c}(y)-\chi_E(y)}{|x_k-y|^{n+s}}\,dy
\right|\le \int_{B_r^c}\frac{dy}{|x_k-y|^{n+s}}\,dy\le C
\int_{B_r^c}\frac{dy}{|y|^{n+s}}\,dy\le\frac{C}{r^s}.$$
The proof of~\eqref{m5cfpa6g823lscgokdmc204r-1} is now completed, thanks to the latter observation and~\eqref{4fGt6hodjf4c619loo3402rfg}. 

The proof of~\eqref{m5cfpa6g823lscgokdmc204r-2} is alike.
\end{proof}

We now recall an easy observation regarding smooth functions:

\begin{lemma}\label{CC1}
Let~$B$ be a ball in~$\R^N$, centered at the origin.
Let~$\varphi_1$, $\varphi_2\in C^{1,\alpha}(B)$, for some~$\alpha\in(0,1]$.

Let
\begin{equation}\label{vasder}\varphi(x_1,\dots,x_N)\coloneqq \begin{dcases}
\varphi_1(x_1,\dots,x_N)&{\mbox{ if $x_1\ge0$,}}\\
\varphi_2(x_1,\dots,x_N)&{\mbox{ if $x_1<0$.}}
\end{dcases}\end{equation}

Assume that, for all~$(0,\dots,x_{N-1},x_{N})\in B$,
\begin{equation}\label{COqwdv} \varphi_1(0,\dots,x_{N-1},x_{N})=\varphi_2(0,\dots,x_{N-1},x_{N})\end{equation}
and, for all~$j\in\{1,\dots,N\}$,
\begin{equation}\label{COqwdv2}\partial_j\varphi_1(0,\dots,x_{N-1},x_{N})=\partial_j\varphi_2(0,\dots,x_{N-1},x_{N}).\end{equation}

Then, $\varphi\in C^{1,\alpha}(B)$.
\end{lemma}

\begin{proof} We stress that~$\varphi$ is continuous in~$B$, thanks to~\eqref{COqwdv}.

We also have that~$\varphi$ is differentiable in~$B$, with
\begin{equation}\label{qowdfjvlb42}
\partial_j\varphi(x_1,\dots,x_{N})=\begin{dcases}\partial_j
\varphi_1(x_1,\dots,x_{N})&{\mbox{ if $x_1\ge0$,}}\\
\partial_j\varphi_2(x_1,\dots,x_{N})&{\mbox{ if $x_1<0$.}}
\end{dcases}
\end{equation}
This follows from~\eqref{vasder} when~$x_1\ne0$, hence we focus on the case~$x_1=0$.

For this, we use~\eqref{COqwdv2} and we see that, as~$h=(h_1,\dots,h_{N})\to0$,
\begin{eqnarray*}&&
\varphi(h_1,x_2+h_2,\dots,x_{N}+h_{N})-\varphi(0,x_2,\dots,x_{N})\\&&\quad=
\begin{dcases}
\varphi_1(h_1,x_2+h_2,\dots,x_{N}+h_{N})-\varphi_1(0,x_2,\dots,x_{N})&{\mbox{ if }}h_1>0,\\
\varphi_2(h_1,x_2+h_2,\dots,x_{N}+h_N)-\varphi_2(0,x_2,\dots,x_{N})&{\mbox{ if }}h_1<0,\end{dcases}\\&&\quad=
\begin{dcases}
\nabla\varphi_1(0,x_2,\dots,x_{N})\cdot h+o(h)
&{\mbox{ if }}h_1>0,\\
\nabla\varphi_2(0,x_2,\dots,x_{N})\cdot h+o(h)
&{\mbox{ if }}h_1<0,
\end{dcases}\\&&\quad=\nabla\varphi_1(0,x_2,\dots,x_{N})\cdot h+o(h),
\end{eqnarray*} from which the proof of~\eqref{qowdfjvlb42} follows.

Now, to complete the proof of~\eqref{CC1}, we check that,
for all~$x=(x_1,\dots,x_N)$ and~$y=(y_1,\dots,y_N)$ in~$B$,
with~$x\ne y$,
\begin{equation}\label{q0wdfoeujvbf}
\frac{|\nabla\varphi(x)-\nabla\varphi(y)|}{|x-y|^\alpha}\le C
\max\left\{
\|\varphi_1\|_{C^{1,\alpha}(B)},\,\|\varphi_2\|_{C^{1,\alpha}(B)}
\right\},
\end{equation}
for some~$C\ge1$.

When~$x_1>0$ and~$y_1>0$, as well as when~$x_1<0$ and~$y_1<0$,
the claim in~\eqref{q0wdfoeujvbf} is a direct consequence of~\eqref{qowdfjvlb42}
and the regularity assumption on~$\varphi_1$ and~$\varphi_2$.

Hence, we can restrict to the case in which~$x_1>0>y_1$. In this case, we pick~$z=(0,z_2,\dots,z_N)$ in the segment joining~$x$ to~$y$ and we remark that~$|x-y|=|x-z|+|z-y|$. Therefore, by~\eqref{qowdfjvlb42},
\begin{eqnarray*}&&
|\nabla\varphi(x)-\nabla\varphi(y)|\le
|\nabla\varphi(x)-\nabla\varphi(z)|+
|\nabla\varphi(z)-\nabla\varphi(y)|\\&&\qquad=|\nabla\varphi_1(x)-\nabla\varphi_1(z)|+
|\nabla\varphi_2(z)-\nabla\varphi_2(y)|\\&&\qquad\le
\|\varphi_1\|_{C^{1,\alpha}(B)}|x-z|^\alpha+
\|\varphi_2\|_{C^{1,\alpha}(B)}|z-y|^\alpha\\&&\qquad\le
\max\left\{
\|\varphi_1\|_{C^{1,\alpha}(B)},\,\|\varphi_2\|_{C^{1,\alpha}(B)}
\right\}\;\left(
|x-z|^\alpha+|z-y|^\alpha
\right)\\&&\qquad\le\max\left\{
\|\varphi_1\|_{C^{1,\alpha}(B)},\,\|\varphi_2\|_{C^{1,\alpha}(B)}
\right\}\;\left(
|x-y|^\alpha+|x-y|^\alpha
\right),
\end{eqnarray*}
and~\eqref{q0wdfoeujvbf} plainly follows.
\end{proof}

With the preliminary work done so far we can now complete the proof of Theorem~\ref{R-wf0jgvhr90tiuohjtktyu}.
	
\begin{proof}[Proof of Theorem~\ref{R-wf0jgvhr90tiuohjtktyu}]
Given~$x_1$, $x_n\in\R$, we let
$$\left(-\frac\pi2,\frac\pi2\right)\ni\vartheta\longmapsto f(\vartheta)\coloneqq 
-x_1\tan\vartheta + x_n$$
and we observe that
\begin{equation}\label{qojdwfvl:ERfgb0-ZSipo}
{\mbox{$f$ is nondecreasing when~$x_1\le0$ and nonincreasing when~$x_1\ge0$.}}
\end{equation}

We claim that when~$\vartheta_1>\vartheta_2$ we have that
\begin{equation}\label{02eorfjgvbg45h54}
\big\{ {\mbox{$x\in B_r$ s.t. $x_1\le0$ and $\omega_1\cdot x<0$}}\big\}\subset \big\{ {\mbox{$x\in B_r$ s.t. $x_1\le0$ and $\omega_2\cdot x<0$}}\big\}
\end{equation}
and when~$\vartheta_1<\vartheta_2$ we have that
\begin{equation}\label{02eorfjgvbg45h54n}\big\{ {\mbox{$x\in B_r$ s.t. $x_1\ge0$ and $\omega_2\cdot x<0$}}\big\}\subset \big\{ {\mbox{$x\in B_r$ s.t. $x_1\ge0$ and $\omega_1\cdot x<0$}}\big\}
.\end{equation}
Indeed, suppose that~$x_1\le0$ and~$\vartheta_1>\vartheta_2$.
Then, since~$\cos\vartheta>0$,
we deduce from~\eqref{qojdwfvl:ERfgb0-ZSipo} that
\begin{eqnarray*}&& \omega_1\cdot x=-x_1\sin\vartheta_1+x_n\cos\vartheta_1
=f(\vartheta_1)\cos\vartheta_1\ge f(\vartheta_2)\cos\vartheta_1\\&&\qquad
=f(\vartheta_2)\cos\vartheta_2\quad\frac{\cos\vartheta_1}{\cos\vartheta_2}=
\omega_2\cdot x
\quad\frac{\cos\vartheta_1}{\cos\vartheta_2}
\end{eqnarray*}
and~\eqref{02eorfjgvbg45h54} follows.

Similarly, if~$\vartheta_1<\vartheta_2$ one obtains~\eqref{02eorfjgvbg45h54n}.

Now we claim that
\begin{equation}\label{wofjvne:23er}
\vartheta_1=\vartheta_2.
\end{equation}
For this, we argue by contradiction.
Namely, suppose that the claim in~\eqref{wofjvne:23er} does not hold.
Then, either~$\vartheta_1>\vartheta_2$ or~$\vartheta_1<\vartheta_2$.
Accordingly, either~\eqref{02eorfjgvbg45h54} holds true (and thus we can use Lemma~\ref{sdfertyu:qwdf}) or~\eqref{02eorfjgvbg45h54n} is satisfied
(and in this case we can rely on footnote~\ref{02eorfjgvbg45h54nb}
on page~\pageref{02eorfjgvbg45h54nb}). In any case,
$$\left|\lim_{(\partial E)\cap\Omega\ni x\to0}\int_{\R^n}\frac{\chi_{E^c}(y)-\chi_E(y)}{|x-y|^{n+s}}\,dy\right|=+\infty.$$
This violates condition~\eqref{eq: s-min surf}
(recall Theorem~\ref{PROP1}) and this proves~\eqref{wofjvne:23er}.

Hence, we set~$\vartheta\coloneqq \vartheta_1=\vartheta_2$ and we have that~$E$ is of class~$C^{1,\alpha}$ in the vicinity of the origin (see Lemma~\ref{CC1}).
As a consequence, to complete the proof of Theorem~\ref{R-wf0jgvhr90tiuohjtktyu},
it remains to check that~$\vartheta=0$.

Suppose not. Then, we can use Lemma~\ref{osjdlcvORmJUfbgb}
and deduce from either~\eqref{m5cfpa6g823lscgokdmc204r-1}
or~\eqref{m5cfpa6g823lscgokdmc204r-2} that
$$\left|\lim_{k\to+\infty}\int_\Omega\frac{\chi_{E^c}(y)-\chi_E(y)}{|x_k-y|^{n+s}}\,dy\right|=+\infty.$$
This is in contradiction with condition~\eqref{eq: free boundary}
(recall Theorem~\ref{PROP1}) and the proof of
Theorem~\ref{R-wf0jgvhr90tiuohjtktyu} is complete.
	\end{proof}
	
\subsection*{Acknowledgments} MB was supported by the European
Research Council under Grant Agreement No 948029. 
SD was supported by the Australian Research Council
Future Fellowship FT230100333 ``New perspectives on nonlocal equations''. EV
was supported by the Australian Laureate Fellowship FL190100081 ``Minimal surfaces, free
boundaries and partial differential equations''.

The authors are grateful to Matteo Cozzi for discussions related to Lemma~\ref{lem: R*}. Most of this work was carried out while MB was visiting the University of Western Australia.

\printbibliography

\end{document}